\RequirePackage{rotating}
\documentclass{my-eectEA}

\usepackage{amssymb}
\usepackage{amsmath}
\usepackage{paralist}
\usepackage[misc]{ifsym}
\usepackage{epsfig}
\usepackage{epstopdf}
\usepackage{caption}
\usepackage{subcaption}
\usepackage[colorlinks=true]{hyperref}
\hypersetup{urlcolor=blue, citecolor=red}

\usepackage{algorithm}
\usepackage{algorithmic}
\usepackage{tikz}
\usetikzlibrary{shapes}

\usepackage{cite}
\usepackage{microtype}

\allowdisplaybreaks

\def\currentvolume{X}
 \def\currentissue{X}
  \def\currentyear{200X}
   \def\currentmonth{XX}
    \def\ppages{X-XX}
     \def\DOI{10.3934/eect.2025018}

\newtheorem{theorem}{Theorem}[section]

\newtheorem{lemma}[theorem]{Lemma}

\theoremstyle{definition}
\newtheorem{definition}[theorem]{Definition}
\newtheorem{remark}[theorem]{Remark}

\title[{\textls[-40]{Fractional SIR model of a reaction-diffusion with Caputo derivative}}]{Fractional 
differential equations of a reaction-diffusion SIR model involving the Caputo-fractional 
time-derivative and a nonlinear diffusion operator}

\author[Achraf Zinihi, Moulay Rchid Sidi Ammi and Delfim F. M. Torres]{}



\thanks{$^*$Corresponding author: Achraf Zinihi}

\begin{document}
\maketitle

\centerline{\scshape
Achraf Zinihi$^{{\href{mailto:a.zinihi@edu.umi.ac.ma;achrafzinihi99@gmail.com}{\textrm{\Letter}}}*1}$,
Moulay Rchid Sidi Ammi$^{{\href{mailto:rachidsidiammi@yahoo.fr}{\textrm{\Letter}}}1}$}
\centerline{\scshape and Delfim F. M. Torres$^{{\href{mailto:delfim@ua.pt}{\textrm{\Letter}}}2}$}

\medskip

{\footnotesize
\centerline{$^1$Department of Mathematics, MAIS Laboratory,} \centerline{AMNEA Group, Faculty of Sciences and Technics,}
\centerline{Moulay Ismail University of Meknes, Errachidia 52000, Morocco}}

\medskip

{\footnotesize
\centerline{$^2$Center for Research and Development in Mathematics and Applications (CIDMA),}
\centerline{Department of Mathematics, University of Aveiro, 3810-193 Aveiro, Portugal}}

\bigskip

\centerline{(Communicated by Arran Fernandez)}

% -------------------------------------------------

\begin{abstract}
The main aim of this study is to analyze a fractional parabolic SIR epidemic 
model of a reaction-diffusion, by using the nonlocal Caputo fractional 
time-fractional derivative and employing the $p$-Laplacian operator.
The immunity is imposed through the vaccination program, which is regarded 
as a control variable. Finding the optimal control pair that reduces the number 
of sick people, the associated vaccination, and treatment expenses across 
a constrained time and space is our main study. The existence and uniqueness 
of the nonnegative solution for the spatiotemporal SIR model are established.
It is also demonstrated that an optimal control exists.
In addition, we obtain a description of the optimal control in terms of state 
and adjoint functions. Then, the optimality system is resolved by a discrete 
iterative scheme that converges after an appropriate test, similar to the 
forward-backward sweep method. Finally, numerical approximations are given 
to show the effectiveness of the proposed control program, which provides 
meaningful results using different values of the fractional order and $p$, 
respectively the order of the Caputo derivative and the $p$-Laplacian operators.
\end{abstract}

% -------------------------------------------------

\section{Introduction}
\label{S1}

Fractional calculus is an ancient subject that arose as a result of a relevant
question that l'H\^opital asked Leibniz in a letter about the possible meaning
of a derivative of order $0.5$ \cite{Leibniz1849letter}. Recently, many
researchers have focused their attention on modeling real phenomena using fractional
derivatives. The dynamics of these problems were modeled and studied using the concept
of fractional derivatives. Such problems arise, for example, in biology, epidemiology,
physics, ecology, engineering, and various other areas of applied science
\cite{kilbas2006,Milici2019}.

In the literature of fractional calculus, systems that use fractional derivatives
can exhibit a more believable behavior \cite{Gorenflo1999, Ngoc2022, Torres2024}. Fractional
derivatives have many different definitions \cite{Gorenflo1999}, the most well-known
being the Riemann--Liouville and the Caputo fractional derivatives. The first is not always
relevant for modeling biological and physical systems, due to the fact
that the fractional derivative of a constant is different from zero. Moreover,
initial conditions of Cauchy problems are represented through fractional derivatives.
In contrast, Caputo derivatives offer an alternative where initial
conditions are presented as in the classical integer-order case and a constant's
fractional derivative is zero \cite{Diethelm2012,Gorenflo1999}.

Fractional derivatives that possess a nonsingular kernel have also generated
a large interest from researchers. This is due to nonsingular memory and
also for not obeying the algebraic standards of bonding and switching
the Mittag--Leffler function. The Caputo--Fabrizio fractional time-derivative
in the Caputo sense is sometimes preferred for modeling biological, physical,
or chemical dynamic systems, by describing diffusion and heterogeneity phenomena
at various scales in a clear and concise manner \cite{Khan2021, Owolabi2021}.

Mathematical models play an important role in modeling epidemiological phenomena
to describe the dynamic behavior of these phenomena, analyzing them, and combating
their spread, especially infectious diseases 
\cite{Bailey1975, Elgart2023, Rao2021, Zinihi2024MM}.
The SIR epidemic model presented by \cite{Kermack1927} is the first
mathematical model and one of the most significant pandemic models
since its birth in 1927. The population in this model is divided
into three compartments for the pathological condition.
In the first category, susceptible individuals $S$ represent
those who have not yet contracted the disease but are at risk of doing so.
Infected individuals $I$ describe the persons who have contracted
the disease and are able to spread it to those who are susceptible.
The third and last compartment are the recovered, denoted by $R$, that represents
individuals who have acquired immunity and become immune to the disease.
Recently, the classic SIR model was adopted to design many
infectious diseases, such as those spread from coronaviruses
\cite{Alshomrani2021,Calvetti2023,Ding2022,Ji2012,Mouaouine2018}.

A series of statistical studies, in several parts of the world,
have confirmed that the movement of individuals has a significant negative
impact on the spread of diseases. One can support this idea by mentioning
the SARS-CoV-2 virus that killed thousand of people, as the COVID-19 pandemic
has led to disturbance problems and crises in most areas in the world \cite{Fu2020}.
A statistical analysis on the Ebola virus in Western Africa has addressed how
the geographical effect has affected the disease's transmission
in this area \cite{Kramer2016}. According to a different study,
the prevalence of H1N1, a subtype of influenza $A$ virus,
an infectious viral disease that causes upper respiratory tract infections and,
in some cases, also lower respiratory tract inflammation, and results
in symptoms such as chills, nasal secretions, low appetite, and fever,
is very sensitive to people's mobility \cite{daCosta2018}.

Today's mobility of individuals is increasingly affecting the spread of diseases,
because of the high urban population density and abundance of transportation options
(high number of means of transport), as a result of scientific and technological
development, and the need for people's mobility for various reasons,
including studying, tourism, trade, work, etc. Therefore, the spatial factor
plays an important role in disease transmission. To include the spatial effect
in our biological system, and taking into account recent works that used PDEs
to model the mobility of people, the symbol $(x, t)$ in our model describes
spatial diffusion on time $t$ and space $x$
\cite{Ducasse2023, Lu2022, SidiAmmi2022, Walker2023, Zinihi2024IM}.

In recent years, there has been a growing interest in incorporating fractional calculus
and nonlinear operators into mathematical models to provide more accurate and realistic
descriptions of complex phenomena. When considering the SIR epidemic model,
utilizing the Caputo fractional time-derivative and the $p$-Laplacian operator can offer several motivations.
Unlike the traditional SIR model, the Caputo ($\mathcal{C}$, for short) fractional derivative allows us
to account for nonlocal and long-range interactions, mirroring the complex and interconnected
nature of real-world epidemics. It introduces memory effects, capturing the influence of past
infections on current transmission probabilities, which is crucial for diseases with varying
infection histories. This approach also accommodates anomalous diffusion patterns often observed
in epidemiology, providing a more versatile framework to model outbreaks. By embracing the
$\mathcal{C}$-fractional derivative, we can better predict long-term epidemic trends,
improve parameter estimation, and contribute to more accurate
and adaptable epidemiological modeling that can help shape effective
public health strategies in an ever-evolving landscape of infectious diseases.
Secondly, the diffusion operator, which is a nonlinear generalization of the Laplace operator,
offers advantages when modeling the interactions between infected and susceptible individuals.
The $p$-Laplacian operator considers the heterogeneity and nonlinearity inherent in epidemic dynamics.
In the context of the SIR model, it allows for a more realistic representation
of the varying susceptibility and infectiousness levels among individuals, leading to a better
understanding of the overall epidemic behavior. Furthermore, the combined use of the
$\mathcal{C}$-fractional time-derivative and the diffusion operator into the SIR
model can provide insights into the long-term behavior and stability of the system.
Fractional
calculus provides a broader perspective on the dynamics, allowing for the analysis of fractional
order differential equations. This can help to reveal novel phenomena and uncover additional
aspects of the epidemic spread that are not captured by traditional integer-order models.
The inclusion of the second operator further enhances the model's ability to capture complex
interactions and nonlinear effects, potentially leading to more accurate predictions
and control strategies for epidemic outbreaks. In summary, the introduction of the
$\mathcal{C}$-fractional time-derivative and the specific diffusion operator into the
SIR epidemic model provides a multitude of compelling motivations.
These modifications enable the consideration of memory effects, capture nonlinear interactions,
and provide a deeper understanding of the long-term behavior of epidemics. By embracing these
advanced mathematical tools, researchers can enhance the accuracy and realism of their models,
ultimately contributing to improved strategies for mitigating and controlling infectious diseases.
Nevertheless, several challenges have arisen during the course of our study. The nonlinearity
introduced by the second operator has made it considerably challenging to establish
the existence of a solution for our problem. Furthermore, we have encountered significant
difficulties in formulating the adjoint system associated with our proposed problem and,
subsequently, in deriving optimality conditions. Additionally, the discretization
of the problem for obtaining numerical results has presented its own set of challenges.

In the existing literature, several studies have explored FDEs or ODEs for epidemic modeling, 
but they typically do not incorporate PDEs to model spatial dynamics. While some works combine 
fractional derivatives with PDEs, they usually involve linear operators, which limit their 
ability to capture the complex, nonlinear dynamics observed in epidemic spread. 
For instance, studies like \cite{Alshomrani2021, Cui2022, Gatto2021, MR4458835, Zinihi2024ES} 
focus on FDE/ODE models, and others, such as \cite{Ge2021, SidiAmmi2022, SidiAmmi2023, Zinihi2024IM} 
discuss PDE-based approaches without the inclusion of fractional time derivatives 
or the $p$-Laplacian operator. In contrast, our approach uniquely combines fractional 
time derivatives with a PDE framework, using the $p$-Laplacian operator to model the 
spatial and temporal dynamics of disease spread more effectively. This combination allows 
us to capture important features such as spatial heterogeneities and memory effects, which 
are not adequately addressed in traditional models that rely solely on ODEs or linear operators. 
Our work, therefore, offers a novel perspective by integrating these components to better 
simulate the real-world spread of infectious diseases.

Due to the diverse nature of these operators, numerous scholars and experts have
contributed to the scientific literature in various domains. These include,
but are not limited to, the fields of biology \cite{SidiAmmi2023, Zinihi2024IM},
quantum theory \cite{SahaRay2017}, and continuum mechanics \cite{Borggaard2011}.
This widespread interest and application in different scientific disciplines
underline the broad and versatile impact of these operators on the research landscape.
As an illustration, it is worth noting that Tao and Li considered a
chemotaxis-Stokes system with slow $p$-Laplacian diffusion, with $\Omega$ a bounded
smooth domain in $\mathbb{R}^3$ \cite{Tao2020}.
Authors in \cite{SidiAmmi2023} analyzed a reaction-diffusion SIR epidemic model,
formulated as a parabolic system of partial differential equations with the integration
of the $p$-Laplacian operator, underscoring the critical role of vaccination
distribution as a means to induce immunity. The central focus of their research
is to develop an optimal control strategy, meticulously designed to curtail both
the spread of infection and the cost associated with vaccination
within a defined spatiotemporal framework.
This leads to many questions to be discussed in this work, such as:
\begin{itemize}
\item[(i)] What relation can be observed between the diffusion
factor $p$ and the fractional order of differentiation $\alpha$?
	
\item[(ii)] During a period of vaccination,
has the spread of the infection been successfully contained?
\end{itemize}
In this work, we introduce a novel approach to address an optimal control
problem involving reaction-diffusion systems. Additionally, we explore its
potential application in the field of surveillance, where we integrate the
spatial behavior of the population and the control term that represents the
vaccination program. This is a powerful tool that can prevent and control
the infection's spread. Our principal motivation is to investigate and analyze
the impact of a vaccination program on infection-related disease transmission,
in the context of our more accurate model that accounts spatial diffusion.

The paper is organized as follows. Crucial definitions related to
$\mathcal{C}$-fractional calculus, the $p$-Laplacian operator, and
their fundamental properties, are given in Section~\ref{S2}.
Section~\ref{S3} introduces the fractional optimal control of the SIR
model in the context of reaction-diffusion while, in Sections~\ref{S4}
and \ref{S5}, we establish the existence and uniqueness of a nonnegative
solution and we prove the existence of an optimal solution for our proposed model.
Furthermore, Section~\ref{S6} is dedicated to establish necessary optimality conditions.
Before closing the present study, insightful numerical approximations that elucidate
the relationship between variations in $p$ and $\alpha$ and the disease's propagation,
are explained in Section~\ref{S7}. Finally, we provide a concise summary of our study
in Section~\ref{S8}.

% -------------------------------------------------

\section{Notations and preliminaries}
\label{S2}

In this section we recall some definitions and properties
about the fractional integral, the $\mathcal{C}$-fractional
time-derivative, and the $p$-Laplacian operator, that are needed in the sequel.
Let $\mathcal{T}>0$, $f\in H^1(0, \mathcal{T})$, $\alpha\in (0,1)$, $p\geq 2$,
and $\Omega$ denote a bounded and fixed domain of $\mathbb{R}^2$
with a smooth boundary referred to as $\partial \Omega$.

\begin{definition}[See \cite{kilbas2006}]
If $\omega \in \mathbb{C}$
is such that $Re(\omega)>0$, then the Gamma function is given by
$$
\Gamma(\omega):=\int_{0}^{+\infty} e^{-t} t^{\omega-1} d t.
$$
\end{definition}

\begin{definition}[See \cite{kilbas2006}]
\label{D1}
\
\begin{itemize}
\item[(a)] The $\mathcal{C}$-fractional derivative of $f$
with base point $0$ of order $\alpha$ is defined at point $t$ by
\begin{equation}\label{E2.1}
{ }^{\mathcal{C}} \mathcal{D}^{\alpha}_{t} f(t)
= \frac{1}{\Gamma(1-\alpha)} \int_{0}^{t} (t - y)^{-\alpha} f^{\prime}(y) d y.
\end{equation}

\item[(b)] The backward $\mathcal{C}$-fractional derivative
with base point $\mathcal{T}$, is defined by
\begin{equation}\label{E2.2}
{ }^{\mathcal{C}}_{\mathcal{T}} \mathcal{D}^{\alpha}_{t} f(t)
= \frac{-1}{\Gamma(1-\alpha)}
\int_{t}^{\mathcal{T}} (y - t)^{-\alpha} f^{\prime}(y) d y.
\end{equation}
\end{itemize}
\end{definition}

\begin{remark}\label{R2}\label{R1}
If we let $\alpha\rightarrow 1$ in \eqref{E2.1},
then we obtain the usual derivative $\partial_t$.
\end{remark}

\begin{definition}[See \cite{kilbas2006}]
\label{D2}
The fractional integral operator with base point $0$ is written as
\begin{equation}
\label{E2.3}
I^{\alpha} f(t) = \frac{1}{\Gamma(\alpha)} \int_{0}^{t} (t - y)^{\alpha-1} f(y) d y.
\end{equation}
\end{definition}

\begin{lemma}[See \cite{Alikhanov2010}]
\label{L1}
Let $\varphi: [0, \mathcal{T}] \longrightarrow L^2(\Omega)$. Assume that
there exists the fractional derivative of $\varphi$ in the Caputo sense. Then
\begin{equation}
\label{E2.4}
\left({ }^{\mathcal{C}} \mathcal{D}^{\alpha}_{t} \varphi(t),
\varphi(t)\right)_{L^2(\Omega)} \geq \frac{1}{2}
{ }^{\mathcal{C}} \mathcal{D}^{\alpha}_{t} \|\varphi(t)\|_{L^2(\Omega)}^2.
\end{equation}
\end{lemma}

\begin{remark}
Let $g\in H^1(0,T)$. If $f\geq g$, then
\begin{equation}\label{E2.5}
I^{\alpha} f(t) \geq I^{\alpha} g(t).
\end{equation}
\end{remark}

\begin{lemma}[See \cite{kilbas2006}]
\label{L2}
One has
\begin{equation}\label{E2.6}
I^{\alpha} \left( { }^{\mathcal{C}}
\mathcal{D}^{\alpha}_{t} f(t)\right) = f(t) - f(0).
\end{equation}
\end{lemma}

The last part of this section is about the $p$-Laplacian operator.
Recall that the constant $k$ may change in each mathematical expression.
Let $\psi$ be a function defined on $\mathbb{R}^2\longrightarrow\mathbb{R}^2$ by
$$
\psi(w) = |w|^{p-2}w, \ \forall w\in\mathbb{R}^2.
$$
The $p$-Laplacian operator is written as
$$
\Delta_p: \quad\begin{array}{l}
W^{1,p}_{0}(\Omega) \ \longrightarrow \ W^{-1,p^\prime}(\Omega)\\
u \quad \longrightarrow \quad \Delta_p u
=\operatorname{div}\left(\psi\left(\nabla u\right)\right):
\quad
\begin{array}{l}
W^{1,p}_{0}(\Omega) \ \longrightarrow \ \mathbb{R} \\
v \quad \longrightarrow \quad\langle\Delta_p u, v\rangle
\end{array},
\end{array}
$$
where $\left\langle \cdot , \cdot\right\rangle$ is the duality product between
$W^{1,p}_{0}(\Omega)$ and $W^{-1,p^\prime}(\Omega)$. Besides, by the Green formula,
and with the homogeneous boundary conditions, we have
$$
\langle -\Delta_p u, v\rangle
= \int_{\Omega} \left| \nabla u \right|^{p-2} \nabla u \cdot \nabla v dx
$$
and
$$
\|\Delta_p u\|_{W^{-1,p^\prime}(\Omega)} = \|u\|_{W^{1,p}_{0}(\Omega)}.
$$

% -------------------------------------------------

\section{Mathematical model}
\label{S3}

Numerous studies in the epidemiological literature have utilized
optimal control theory to examine how vaccination and treatment
affect the spread of infectious diseases \cite{MR4458835, MR4498607, MR4468429}.
The mentioned strategies are used to reduce the spread of the pandemic or to
gain lasting immunity in individuals, by analyzing likely the consequences of
the vaccine given to vulnerable individuals and the treatment provided
to infected people. We assume that the total population $N$ comprises
three subgroups: Susceptible, Infectious, and Removed individuals.
Having introduced these three compartments, a reaction-diffusion
SIR epidemic model can be constructed as follows:
\begin{equation}
\label{eq:SIR:syst}
\left\{\begin{array}{l}
{ }^{\mathcal{C}} \mathcal{D}^{\alpha}_{t}S
= \lambda_1 \Delta_p S+\beta N-\mu S I - \xi S, \vspace{0.2cm}\\
{ }^{\mathcal{C}} \mathcal{D}^{\alpha}_{t}I
= \lambda_2 \Delta_p I+\mu S I-(\xi + \kappa) I,
\qquad (x, t) \in \Omega_\mathcal{T}, \vspace{0.2cm}\\
{ }^{\mathcal{C}} \mathcal{D}^{\alpha}_{t}R
= \lambda_3 \Delta_p R + \kappa I - \xi R,
\end{array}\right.
\end{equation}
where ${ }^{\mathcal{C}} \mathcal{D}^{\alpha}_{t}$ is the
$\mathcal{C}$-fractional derivative with base point $0$
of order $\alpha$ and $\mathcal{T}>0$. Table~\ref{Tab1}
provides the transmission coefficients
relevant to the SIR model.
\begin{table}[h]
\centering
\caption{Transmission coefficients.}\label{Tab1}\vspace{0.1cm}
\begin{tabular}{|c|c||c|c|}
\hline
$ \ \beta \ $ & \text{Birth rate}
& $ \ \xi \ $ & \text{Natural mortality rate}\\ \hline
$ \ \mu \ $ & \text{Effective contact rate}
& $ \ \kappa \ $ & \text{Removal rate}\\ \hline
\end{tabular}
\end{table}

The term $\mu SI$ describes the number of individuals at risk
of contracting the disease (newly infected people) that are being
infected by persons from category $I$, and $\kappa I$ models
the infected individuals that recover without intervention.
As highlighted in Section~\ref{S1}, it is crucial to consider
the spatial dimension when studying disease propagation. To describe
this effect, we introduce specific terms that characterize spatial diffusion.
This will involve the inclusion of diffusion operators, denoted as
$\Delta_p S$, $\Delta_p I$, and $\Delta_p R$, with $\lambda_1,\lambda_2,\lambda_3 >0$
representing the diffusion coefficients for the susceptible, infected,
and recovered populations, respectively.
Taking into account the assumptions outlined above,
we get the mathematical model \eqref{eq:SIR:syst}.

In the sequel we introduce a control function $u$ (vaccination),
assuming that all susceptible individuals that are vaccinated are
transferred directly to the category that has been removed,
as depicted in Figure~\ref{F1}.
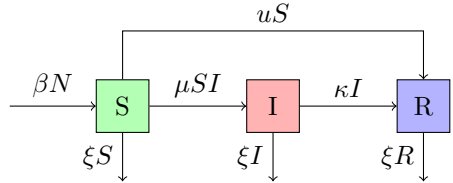
\begin{figure}[H]
\centering
\begin{tikzpicture}[node distance=2cm]
\node (S) [rectangle, draw, minimum size=0.7cm, fill=green!30] {S};
\node (I) [rectangle, draw, minimum size=0.7cm, fill=red!30, right of=S] {I};
\node (R) [rectangle, draw, minimum size=0.7cm, fill=blue!30, right of=I] {R};
\draw[->] (-1.5,0) -- ++(S) node[midway,above]{$\beta N$};
\draw[->] (S.south) -| (0,-1) node[near end,left]{$\xi S$};
\draw [->] (S) -- (I) node[midway,above]{$\mu SI$};
\draw[->] (I.south) -| (2,-1) node[near end,left]{$\xi I$};
\draw [->] (I) -- (R) node[midway,above]{$\kappa I$};
\draw[->] (R.south) -| (4,-1) node[near end,left]{$\xi R$};
\draw [->] (S.north) -| (0,1) -- (4,1) node[midway,above]{$u S$} -| (R.north) ;
\end{tikzpicture}
\caption{Visual representation of the control system \eqref{E3.1}.}\label{F1}
\end{figure}
Mathematically, we obtain the control system defined by
\begin{equation}
\label{E3.1}
\left\{\begin{array}{l}
{ }^{\mathcal{C}} \mathcal{D}^{\alpha}_{t}S
= \lambda_1 \Delta_p S + \beta N-\mu S I- \xi S - uS, \vspace{0.2cm}\\
{ }^{\mathcal{C}} \mathcal{D}^{\alpha}_{t}I
= \lambda_2 \Delta_p I+\mu S I-(\xi + \kappa) I, \vspace{0.2cm}\\
{ }^{\mathcal{C}} \mathcal{D}^{\alpha}_{t}R
= \lambda_3 \Delta_p R + \kappa I- \xi R + uS,
\end{array}\right.
\text{in } \Omega_\mathcal{T}
= \Omega \times [0, \mathcal{T}],
\end{equation}
with the Neumann boundary conditions
\begin{equation}
\label{E3.2}
\nabla S\cdot\vec{n} = \nabla I\cdot\vec{n}
= \nabla R\cdot\vec{n} = 0, \ \text{ on }
\partial\Omega_\mathcal{T}
= \partial\Omega \times [0, \mathcal{T}].
\end{equation}
These conditions are influenced by the normal $\vec{n}$
to the boundary $\partial\Omega$, and
\begin{equation}
\label{E3.3}
S(x, 0)=S_0, \ \ I(x, 0)=I_0, \ \ R(x, 0)=R_0, \ \text{ in }  \Omega.
\end{equation}
According to the no-flux boundary conditions, there is no emigration,
and system \eqref{E3.1}--\eqref{E3.3} is self-contained and features all
dynamic interactions across borders, while the initial conditions are
considered positive for biological reasons.

We now introduce an optimal control problem: to minimize
the objective functional
\begin{equation}
\label{E3.4}
\mathcal{J}(S, I, R, u)
= \|I(\cdot, \mathcal{T})\|_{L^{2}(\Omega)}^{2}
+ \|I\|_{L^{2}(\Omega_\mathcal{T})}^{2}
+ \eta \|u\|_{L^{2}(\Omega_\mathcal{T})}^{2},
\end{equation}
where $\eta$ is a weight constant used to control the vaccination rate, and
\begin{equation}
\label{E3.5}
u\in U_{a d}=\left\{v \in L^{\infty}(\Omega_\mathcal{T})
/ \|v\|_{L^{\infty}(\Omega_\mathcal{T})}<1 \text { and } v>0\right\}.
\end{equation}
We choose the norms $\|\cdot\|_{L^{2}(\Omega_\mathcal{T})}$
and $\|\cdot\|_{L^{2}(\Omega)}$ in \eqref{E3.4} because
the associated spaces are Hilbert spaces, which will allow us to explore,
in Section~\ref{S6}, the differentiability aspects of $\mathcal{J}$ in G\^ateaux sense,
based on the scalar products defined in the previously introduced Hilbert spaces.

Let $\nu = (\nu_1, \nu_2, \nu_3) = (S, I, R)$,
$\nu^0 =( \nu^0_1 ,\nu^0_2, \nu^0_3)=(S_0, I_0, R_0)$,
and $\lambda = (\lambda_1,\lambda_2,\lambda_3)$. Put
$\mathcal{W}(\Omega) = \left(W^{1,p}_{0}(\Omega)\right)^3$
with $\mathcal{W}^\prime(\Omega)$ its dual and
$\mathcal{H}(\Omega) = \left(L^2(\Omega)\right)^3$.
We consider the vector function $\Psi$ defined by
\[
\Psi(\nu(t))=\left(\Psi_1(\nu(t)),\Psi_2(\nu(t)),\Psi_3(\nu(t))\right),
\]
with $\nu(t)(\cdot)=\nu(t,\cdot)$ and
\begin{equation}
\label{E3.6}
\begin{cases}
\Psi_1(\nu(t))
= \beta(\nu_1+\nu_2+\nu_3)-\mu \nu_1\nu_2-(\xi + u)\nu_1,\\
\Psi_2(\nu(t))
= \mu \nu_1\nu_2-(\xi + \kappa)\nu_2,\qquad\qquad\qquad\qquad t\in [0, \mathcal{T}],\\
\Psi_3(\nu(t))
= \kappa \nu_2- \xi \nu_3+u\nu_1.
\end{cases}
\end{equation}
Let $\mathcal{L}$ be the nonlinear operator
\begin{equation}
\label{E3.7}
\begin{gathered}
\mathcal{L}: \quad \begin{array}{l}
D(\mathcal{L})\subset \mathcal{W}(\Omega)
\rightarrow \mathcal{W}^\prime(\Omega) \\
y \quad \rightarrow \quad -\lambda\Delta_p y
\end{array},\\
D(\mathcal{L})=\Big\{\omega\in \mathcal{W}(\Omega) / \
\nabla \omega_i\cdot\vec{n} = 0 \text{ on } \partial\Omega,  i=1,2,3\Big\}.
\end{gathered}
\end{equation}
Then \eqref{E3.1}--\eqref{E3.3} can be rewritten in the form
\begin{equation}
\label{E3.8}
\begin{cases}
{ }^{\mathcal{C}} \mathcal{D}^{\alpha}_{t}\nu(t) + \mathcal{L} \nu(t) = \Psi(\nu(t)),\\
\nu(0)=\nu^0.
\end{cases}
\qquad t\in[0, \mathcal{T}].
\end{equation}

% -------------------------------------------------

\section{Existence and uniqueness of solution}
\label{S4}

In this section, we establish the existence and uniqueness of a solution
to the nonlinear fractional system \eqref{E3.1}--\eqref{E3.3}.
For that we apply the Faedo--Galerkin method.

To prove the existence of a solution for system \eqref{E3.8},
we commence by introducing the crucial notion of weak solution.

\begin{definition}
\label{D3}
We say that $\nu$ is a weak solution for problem \eqref{E3.8}
if it satisfies the initial condition $\nu(0)=\nu^0$
and the following equality:
\[
\left\langle { }^{\mathcal{C}} \mathcal{D}^{\alpha}_{t}\nu,
\varphi \right\rangle + \left\langle \mathcal{L}\nu,
\varphi \right\rangle - \left\langle \Psi(\nu) , \varphi \right\rangle = 0,
\ \forall\varphi\in L^p(0, \mathcal{T};\mathcal{W}(\Omega)).
\]
\end{definition}

\begin{lemma}[See \cite{Jarad2020}]
\label{L4}
If $\nu, \varphi \in C^{\infty}\left(\Omega_\mathcal{T}\right)$, then
$$
\begin{aligned}
\int_{0}^{\mathcal{T}}({ }^{\mathcal{C}} \mathcal{D}^{\alpha}_{t} \nu) \ \varphi d t
=& \int_{0}^{\mathcal{T}} \nu \ ({ }^{\mathcal{C}}_{\mathcal{T}}
\mathcal{D}^{\alpha}_{t} \varphi)  d t
+ \frac{1}{\Gamma(1-\alpha)} \varphi(x, \mathcal{T})
\int_{0}^{\mathcal{T}} \nu(t) (\mathcal{T}-t)^{-\alpha} d t\\
&- \frac{1}{\Gamma(1-\alpha)} \nu(x, 0)
\int_{0}^{\mathcal{T}} \varphi(t) t^{-\alpha} d t.
\end{aligned}
$$
\end{lemma}

Our first result asserts the existence of a weak solution for problem \eqref{E3.8}.

\begin{theorem}
\label{T1}
Let $\nu^0\in  L^2(\Omega)$. Problem \eqref{E3.8} possesses at least
one weak solution within the spaces $L^\infty(0, \mathcal{T};
\mathcal{H}(\Omega))\cap L^p(0, \mathcal{T}; \mathcal{W}(\Omega))$.
\end{theorem}

\begin{proof}
Consider a basis $\{e_j\}_{j\in\mathbb{N}}$ belonging to the space
$\mathcal{W}(\Omega)$, where each $e_j$ is a smooth function defined
in $\Omega$. We aim to prove that the finite-dimensional approximation
solutions $\left(\vartheta_m(t)\right)_{m\in\mathbb{N}}$
converge to the exact one, where
\[
\vartheta_m(t)=\sum_{j=1}^m \omega_j(t) e_j.
\]
These approximated solutions correspond
to the solutions of the following approximate system:
\begin{equation}\label{E4.1}
\begin{cases}
\left\langle { }^{\mathcal{C}} \mathcal{D}^{\alpha}_{t}\vartheta_m,
e_j\right\rangle + \left\langle \mathcal{L}\vartheta_m , e_j\right\rangle
= \left\langle \Psi(\vartheta_m), e_j\right\rangle, \quad t>0, \ j=1, \ldots, m,\\
\vartheta_m(0)=\sum_{j=1}^m \omega_j(0) e_j =:\vartheta_m^0,
\end{cases}
\end{equation}
where $\omega_j(t)=\omega_{m}^j(t)$. Since $\mathcal{L}$ is a bounded operator
and the determinant $\underset{j\neq l}{\det}(\left\langle e_j, e_l\right\rangle)$
is nonzero (due to $\{e_j\}_{j\in\mathbb{N}}$ forming a basis of $\mathcal{W}(\Omega)$),
coupled with the continuity of $\Psi$ as a function, it follows, based on the results of
\cite{Almeida2017,Diethelm2010}, that problem \eqref{E4.1} possesses at least one
local solution within the interval $[0, T_m)$, where $\mathcal{T}_m>0$.
In addition, \eqref{E4.1} can be rewritten in the form
\begin{equation}
\label{E4.2}
{ }^{\mathcal{C}} \mathcal{D}^{\alpha}_{t}\omega_k(t)
= -\lambda \int_{\Omega} |\nabla \vartheta_m|^{p-2} \nabla \vartheta_m
\cdot \nabla e_k dx + \int_{\Omega} \Psi\left(\vartheta_m\right) e_k d x =: G_k(t)
\end{equation}
with $1\leq k\leq m$. By applying the fractional integral \eqref{E2.6} to \eqref{E4.2},
we obtain
\begin{equation}
\label{E4.3}
\omega_k(t) = \omega_k(0) + \frac{1}{\Gamma(\alpha)} \int_0^t (t-s)^{\alpha - 1} G_k(s)d s.
\end{equation}
We proceed by showing the extensibility of this approximate solution over the entire
temporal interval $[0, \mathcal{T}]$. To achieve this, we multiply equation \eqref{E4.1}
by $\omega_i(t)$, sum up over $i=1, \ldots, m$, and employ the Green formula
alongside \eqref{E2.4}, yielding
\begin{equation}
\label{E4.4}
\frac{1}{2} { }^{\mathcal{C}} \mathcal{D}^{\alpha}_{t}
\left\|\vartheta_m\right\|_{\mathcal{H}(\Omega)}^2
+ \lambda\int_{\Omega}\left|\nabla \vartheta_m\right|^{p} d x
\leq \int_{\Omega} \Psi\left(\vartheta_m\right) \vartheta_m d x.
\end{equation}
Using \eqref{E2.5} and \eqref{E2.6}, we get
\begin{equation}
\label{E4.5}
\begin{aligned}
\left\|\vartheta_m(t)\right\|_{\mathcal{H}(\Omega)}^2
\leq &\left\|\vartheta_m(0)\right\|_{\mathcal{H}(\Omega)}^2
- \frac{2\lambda}{\Gamma(\alpha)} \int_{0}^{t}
\int_{\Omega} (t - s)^{\alpha - 1} \left|\nabla \vartheta_m(s)\right|^{p} dx ds\\
&+ \frac{2}{\Gamma(\alpha)} \int_{0}^{t} \int_{\Omega}
(t - s)^{\alpha - 1} \Psi\left(\vartheta_m(s)\right) \vartheta_m(s) dx ds.
\end{aligned}
\end{equation}
Starting from equation \eqref{E4.1}, we derive the following relationship:
$$
{ }^{\mathcal{C}} \mathcal{D}^{\alpha}_{t} N_m
= \mathcal{L}\vartheta_m^1 + \mathcal{L}\vartheta_m^2
+ \mathcal{L}\vartheta_m^3 + \beta N_m - \xi N_m,
$$
where $N_m = \vartheta_m^1 + \vartheta_m^2 + \vartheta_m^3$. Given that
$\mathcal{L}$ is a bounded operator, we can establish the existence
of a constant $C$ such that
$$
{ }^{\mathcal{C}} \mathcal{D}^{\alpha}_{t}
\|N_m\|_{L^2(\Omega)} \leq (C + \beta - \xi) \|N_m\|_{L^2(\Omega)}.
$$
Utilizing \eqref{E2.5} and \eqref{E2.6}, we can obtain that
$$
\|N_m\|_{L^2(\Omega)} \leq \|N_m^0\|_{L^2(\Omega)}
+ \frac{C + \beta - \xi}{\Gamma(\alpha)}
\int_0^t (t-s)^{\alpha-1}\|N_m(s)\|_{L^2(\Omega)} ds.
$$
Applying the Gronwall inequality, we conclude with
$$
\begin{aligned}
\|N_m\|_{L^2(\Omega)}
&\leq \|N_m^0\|_{L^2(\Omega)} \exp\left[ \frac{C + \beta - \xi}{\Gamma(\alpha)}
\int_0^t (t-s)^{\alpha-1} ds\right]\\
&\leq \|N_m^0\|_{L^2(\Omega)} \exp\left[
\frac{C + \beta - \xi}{\Gamma(\alpha)} \frac{t^\alpha}{\alpha} \right] < \infty.
\end{aligned}
$$
This establishes the boundedness of $(\vartheta_m)_m$ within $\mathcal{H}(\Omega)$.
In other words, the second term $\Psi$ is a member of $\mathcal{H}(\Omega)$
with respect to $t\in [0, T_m)$. Consequently, by \eqref{E4.5}, there exist
positive constants $c_1$ and $c_2$ that allow us to express the inequality
\begin{equation}
\label{E4.6}
\left\|\vartheta_m(t)\right\|_{\mathcal{H}(\Omega)}^2
\leq  \left\|\vartheta_m(0)\right\|_{\mathcal{H}(\Omega)}^2 + c_1 t + c_2.
\end{equation}
Accordingly, we can confidently broaden the scope of our approximate solution
to encompass the entirety of the temporal interval $[0, \mathcal{T}]$.
As a result, we obtain that
\begin{equation}\label{E4.7}
\left\|\vartheta_m\right\|_{L^\infty(0,
\mathcal{T}; \mathcal{H}(\Omega))} \leq  c.
\end{equation}
Given that $W^{1,p}_0(\Omega)$ is compactly embedded in $L^p(\Omega)\subseteq L^2(\Omega)$,
we deduce that the sequence $(\vartheta_m^i)_{i=1,2,3}$ is also compact within $L^2(\Omega)$.
Considering now the inequality \eqref{E4.7}, and knowing that $\mathcal{L}$
is a bounded operator, we establish the existence of a subsequence $(\vartheta_m)$,
denoted again as $(\vartheta_m)$, such that the following weak convergence relations hold:
\begin{equation}
\label{E4.8}
\begin{aligned}
& \vartheta_m \rightharpoonup \nu \text { weak-star in }
L^{\infty}\left(0, \mathcal{T}; \mathcal{H}(\Omega)\right), \\
& \mathcal{L} \vartheta_m \rightharpoonup  \chi \text { weakly in }
L^{p^\prime}\left(0, \mathcal{T}; \mathcal{W}^\prime(\Omega)\right),\\
& \vartheta_m \rightharpoonup  \nu \text { weakly in }
L^p(0, \mathcal{T}; \mathcal{W}(\Omega)),\\
& \vartheta_m \longrightarrow  \nu \text { strongly in }
\mathcal{H}(\Omega_\mathcal{T}), \\
& \vartheta_m \longrightarrow  \nu \text { a.e. in }
\mathcal{H}(\Omega_\mathcal{T}).
\end{aligned}
\end{equation}
Expressing $\vartheta_m^1 \vartheta_m^2 - \nu_1 \nu_2$ as
$\left(\vartheta_m^1 - \nu_1\right) \vartheta_m^2
+ \nu_1\left(\vartheta_m^2 - \nu_2\right)$, and considering
the convergences \eqref{E4.8}, we can utilize \eqref{E4.7}
to conclude that $\vartheta_m^1 \vartheta_m^2 \rightharpoonup
\nu_1 \nu_2$ weakly in $L^p(0, \mathcal{T}; W^{1, p}_0(\Omega))$.
This implies that
\begin{equation}
\label{E4.9}
\Psi(\vartheta_m) \rightharpoonup  \Psi(\nu)
\text { weakly in } L^p(0, \mathcal{T}; \mathcal{W}(\Omega)).
\end{equation}
Note that the space $D^\prime(\Omega_\mathcal{T})$ is the dual of
$C_0^\infty(\Omega_\mathcal{T})$. If $\varphi\in C_0^\infty(\Omega_\mathcal{T})$, then
$$
\int_{0}^{\mathcal{T}} \int_{\Omega} \vartheta_m  \ ({ }^{\mathcal{C}}_{\mathcal{T}}
\mathcal{D}^{\alpha}_{t} \varphi) dx dt \longrightarrow \int_{0}^{\mathcal{T}}
\int_{\Omega} \nu \ ({ }^{\mathcal{C}}_{\mathcal{T}} \mathcal{D}^{\alpha}_{t} \varphi) dx dt,
$$
$$
\int_{\Omega} \varphi(x, \mathcal{T}) \int_{0}^{\mathcal{T}}
\vartheta_m(t) (\mathcal{T}-t)^{-\alpha} dt dx
\longrightarrow \int_{\Omega} \varphi(x, \mathcal{T})
\int_{0}^{\mathcal{T}} \nu(t) (\mathcal{T}-t)^{-\alpha} dt dx,
$$
and
$$
\frac{1}{\Gamma(1-\alpha)} \vartheta_m(x, 0) \int_{0}^{\mathcal{T}}
\varphi(t) t^{-\alpha} d t \longrightarrow \frac{1}{\Gamma(1-\alpha)}
\nu(x, 0) \int_{0}^{\mathcal{T}} \varphi(t) t^{-\alpha} d t.
$$
By Lemma~\ref{L4}, and according to the initial conditions of \eqref{E4.1},
we find that
\begin{equation}
\label{E4.10}
{ }^{\mathcal{C}} \mathcal{D}^{\alpha}_{t}\vartheta_m
\rightharpoonup { }^{\mathcal{C}} \mathcal{D}^{\alpha}_{t}
\nu \text{ weakly in } D^\prime(\Omega_\mathcal{T}).
\end{equation}
The remaining step is to prove that $\mathcal{L}\nu = \chi$. Recall
that $\mathcal{L}$ is of type $M$, that is, if $\vartheta_m\rightharpoonup \nu$,
$\mathcal{L}\vartheta_m\rightharpoonup \chi$, and $\limsup_{m\rightarrow\infty}
\langle\mathcal{L}\vartheta_m , \vartheta_m\rangle\leq \chi(\nu)$,
then $\mathcal{L}\nu = \chi$. Hence, we have
\[
\limsup _{m\rightarrow\infty} \left\langle\mathcal{L}\vartheta_m,
\vartheta_m\right\rangle \leq \left\langle\chi , \nu\right\rangle.
\]
Due to the monotonicity of $\mathcal{L}$, we can observe that
for all $\vartheta_m,\omega\in D(\mathcal{L})$
$$
\begin{aligned}
D_m &= \int_{0}^{\mathcal{T}} \left\langle\mathcal{L}\vartheta_m
- \mathcal{L}\omega , \vartheta_m - \omega\right\rangle dt\\
&= \int_{0}^{\mathcal{T}} \left\langle\mathcal{L}\vartheta_m,
\vartheta_m\right\rangle dt - \int_{0}^{\mathcal{T}}
\left\langle\mathcal{L}\vartheta_m, \omega\right\rangle dt
- \int_{0}^{\mathcal{T}} \left\langle\mathcal{L}\omega ,
\vartheta_m - \omega\right\rangle dt \geq 0.
\end{aligned}
$$
Therefore, we conclude that
$$
\begin{aligned}
\limsup _{m\rightarrow\infty} D_m
&\leq \int_{0}^{\mathcal{T}} \left\langle\chi , \nu\right\rangle dt
- \int_{0}^{\mathcal{T}} \left\langle\chi , \omega\right\rangle dt
- \int_{0}^{\mathcal{T}} \left\langle\mathcal{L}\omega , \nu - \omega\right\rangle dt\\
&\leq \int_{0}^{\mathcal{T}} \left\langle\chi
- \mathcal{L}\omega, \nu - \omega\right\rangle dt.
\end{aligned}
$$
Finally, we establish that
$$
\int_{0}^{\mathcal{T}} \left\langle\chi
- \mathcal{L}\omega , \nu - \omega\right\rangle dt\geq 0.
$$
Considering $\delta>0$, we set  $\omega = \nu - \delta h \in D(\mathcal{L})$.
Since $\mathcal{L}$ is hemicontinuous, that is, the function $t\mapsto \left\langle
\mathcal{L}\left( \nu_1 + t \nu_2\right) , \nu_3 \right\rangle$ is continuous
of $\mathbb{R}\rightarrow\mathbb{R}$ for all $\nu_1, \nu_2, \nu_3\in W^{1, p}_0(\Omega)$,
we can state that
$$
\delta\int_{0}^{\mathcal{T}} \left\langle\chi
- \mathcal{L}(\nu - \delta h) , h\right\rangle dt \geq 0.
$$
Consequently,
$$
\left\langle\chi - \mathcal{L}(\nu - \delta h) , h\right\rangle \geq 0.
$$
As we let $\delta\rightarrow 0$, we find that
$$
\left\langle\chi - \mathcal{L}\nu , h\right\rangle
\geq 0, \ \forall h \in D(\mathcal{L}),
$$
and, subsequently,
\begin{equation}
\label{E4.11}
\chi = \mathcal{L}\nu.
\end{equation}
The proof is complete.
\end{proof}

We now prove uniqueness of solution.

\begin{theorem}
\label{T2}
A weak solution $\nu$ of \eqref{E3.8} is unique.
\end{theorem}

\begin{proof}
Consider two weak solutions to problem \eqref{E4.1},
$\upsilon$ and $\bar{\upsilon}$. Let us denote their
difference by $\omega = \upsilon - \bar{\upsilon}$. We have
\[
{ }^{\mathcal{C}} \mathcal{D}^{\alpha}_{t} \omega
+ \left(\mathcal{L}\upsilon - \mathcal{L}\bar{\upsilon}\right)
= \left(\Psi(\upsilon) - \Psi(\bar{\upsilon})\right),
\]
with its associate initial data $\omega^0 = \upsilon^0 - \bar{\upsilon}^0$.
Multiplying this equation by $\omega$, we obtain that
\[
\left\langle { }^{\mathcal{C}} \mathcal{D}^{\alpha}_{t} \omega,
\omega\right\rangle +  \left\langle\mathcal{L}\upsilon
- \mathcal{L}\bar{\upsilon} , \upsilon - \bar{\upsilon}\right\rangle
= \left\langle \Psi(\upsilon) - \Psi(\bar{\upsilon}),
\upsilon - \bar{\upsilon}\right\rangle.
\]
As stated in Theorem~\ref{T1}, $\nu$ is constrained within the space
$\mathcal{H}(\Omega)$. Thus, function $\Psi$ is Lipschitz continuous
in $\nu$ uniformly with respect to $t\in [0, \mathcal{T}]$
and the operator $\mathcal{L}$ is monotone
\cite{Barbu1994,SidiAmmi2022,SidiAmmi2023}.
By \eqref{E2.4}, we get
\[
{ }^{\mathcal{C}} \mathcal{D}^{\alpha}_{t}
\left\|\omega(t)\right\|_{\mathcal{H}(\Omega)}^2
\leq C  \left\|\omega(t)\right\|_{\mathcal{H}(\Omega)}^2.
\]
Using \eqref{E2.5} and \eqref{E2.6}, we have
\[
\left\|\omega(t)\right\|_{\mathcal{H}(\Omega)}^2
\leq \left\|\omega^0\right\|_{\mathcal{H}(\Omega)}^2
+ \frac{C}{\Gamma(\alpha)} \int_{0}^{t}
(t - s)^{\alpha-1}\left\|\omega(s)\right\|_{\mathcal{H}(\Omega)}^2 ds.
\]
According to Gronwall's inequality, we conclude that
\[
\left\|\omega(t)\right\|_{\mathcal{H}(\Omega)}^2
\leq \left\|\omega^0\right\|_{\mathcal{H}(\Omega)}^2
\exp\left[ \frac{C}{\Gamma(\alpha)} \int_0^t (t-s)^{\alpha-1} ds\right],
\]
which completes the proof.
\end{proof}

\begin{theorem}\label{T3}
The solution of \eqref{E3.1}--\eqref{E3.3} is positive.
\end{theorem}

\begin{proof}
We begin by establishing the nonnegativity of $I$.
Let us express $I$ as the difference of its positive
and negative components:
$$
I = I^+ - I^-,
$$
where
$$
\begin{array}{l}\vspace{0.15cm}
I^+(x, t) = \sup\{I(x, t), 0\}
\ \text{ and } \ I^-(x, t) = \sup\{-I(x, t), 0\}.
\end{array}
$$
Now, by multiplying the second equation of the system \eqref{E3.1}--\eqref{E3.3}
by $I^-$, and integrating over $\Omega$, we can use \eqref{E2.4}
to derive the inequality
$$
\frac{1}{2} { }^{\mathcal{C}} \mathcal{D}^{\alpha}_{t}
\left\|I^-\right\|_{L^2(\Omega)}^2 + \lambda\int_{\Omega}
\left|\nabla I\right|^{p-2} \left(\nabla I^-\right)^{2} d x
\leq \mu\int_{\Omega} S (I^-)^2 dx + (\xi+\kappa)\int_{\Omega} (I^-)^2 dx.
$$
In simpler terms,
$$
\frac{1}{2}^{\mathcal{C}} \mathcal{D}^{\alpha}_{t}\left\|I^-\right\|_{L^2(\Omega)}^2
\leq \mu\int_{\Omega} S (I^-)^2 dx \leq \mu N\int_{\Omega} (I^-)^2 dx.
$$
Utilizing \eqref{E2.5} and \eqref{E2.6}, we can further simplify the inequality to
$$
\left\|I^-\right\|_{L^2(\Omega)}^2 \leq \left\|I_0^-\right\|_{L^2(\Omega)}^2
+ \frac{2\mu N}{\Gamma(\alpha)}\int_{0}^{t}
(t - y)^{\alpha-1} \left\|I^-(y)\right\|_{L^2(\Omega)}^2 dy.
$$
By applying Gronwall's inequality, we deduce that
$$
\left\|I^-\right\|_{L^2(\Omega)}^2 \leq \left\|I_0^-\right\|_{L^2(\Omega)}^2
\exp\left[ \frac{2\mu N}{\Gamma(\alpha)} \int_0^t (t-y)^{\alpha-1} dy\right].
$$
This implies that $I^-$ must be zero and, therefore, $I$ is positive.
The same reasoning applies to $R$. By multiplying the third equation
of \eqref{E3.1}--\eqref{E3.3} by $R^-$, one obtains
$$\begin{aligned}
&\frac{1}{2}^{\mathcal{C}} \mathcal{D}^{\alpha}_{t}
\left\|R^-\right\|_{L^2(\Omega)}^2 + \lambda\int_{\Omega}
\left|\nabla R\right|^{p-2} \left(\nabla R^-\right)^{2} d x
\\&\leq \kappa\int_{\Omega} I R^- dx
- \xi\int_{\Omega} (R^-)^2 dx + \int_{\Omega} u S R^- dx.
\end{aligned}
$$
This can be simplified as
$$
{ }^{\mathcal{C}} \mathcal{D}^{\alpha}_{t}\left\|R^-\right\|_{L^2(\Omega)}^2
\leq 2\kappa\int_{\Omega} I R^- dx + 2\int_{\Omega} u S R^- dx.
$$
As $u$ belongs to $U_{ad}$, by the Cauchy--Schwarz inequality,
there exists a constant $C$ such that
$$
{ }^{\mathcal{C}} \mathcal{D}^{\alpha}_{t}\left\|R^-\right\|_{L^2(\Omega)}^2
\leq 2\kappa \left\|I\right\|_{L^2(\Omega)} \left\|R^-\right\|_{L^2(\Omega)}
+ 2C\left\|S\right\|_{L^2(\Omega)} \left\|R^-\right\|_{L^2(\Omega)}.
$$
Because $L^2(\Omega) \subset L^1(\Omega)$, there exists a constant $C$ such that
$$
{ }^{\mathcal{C}} \mathcal{D}^{\alpha}_{t}\left\|R^-\right\|_{L^2(\Omega)}^2
\leq C \left( \left\|I\right\|_{L^2(\Omega)}^2
+ \left\|S\right\|_{L^2(\Omega)}^2 \right) \left\|R^-\right\|_{L^2(\Omega)}^2.
$$
Now, with the help of \eqref{E2.5} and \eqref{E2.6}, we can state that
$$\begin{aligned}
&\left\|R^-\right\|_{L^2(\Omega)}^2 \\& \leq \left\|R_0^-\right\|_{L^2(\Omega)}^2
\\& \quad + \frac{C}{\Gamma(\alpha)}\int_{0}^{t} (t - y)^{\alpha-1}
\left(\left\|I(y)\right\|_{L^2(\Omega)}^2
+ \left\|S(y)\right\|_{L^2(\Omega)}^2 \right)
\left\|R^-(y)\right\|_{L^2(\Omega)}^2 dy.
\end{aligned}
$$
Therefore,
$$\begin{aligned}
&\left\|R^-\right\|_{L^2(\Omega)}^2 \\& \leq \left\|R_0^-\right\|_{L^2(\Omega)}^2
\exp\left[\frac{C}{\Gamma(\alpha)}\int_{0}^{t} (t - y)^{\alpha-1}
\left(\left\|I(y)\right\|_{L^2(\Omega)}^2
+ \left\|S(y)\right\|_{L^2(\Omega)}^2 \right)dy\right].
\end{aligned}
$$
This leads to the conclusion that $R^-$ must be zero and, consequently,
$R$ is nonnegative. Using the same methodology, we can establish
that $S$ is also positive.
\end{proof}

% -------------------------------------------------

\section{Existence of an optimal solution}
\label{S5}

To demonstrate the existence of an optimal control
using the technique of minimizing sequences
we make use of Lemma~\ref{L5}. Before, we recall
the definition of the Mittag--Leffler function $E_{\alpha}$.
Let $\omega \in \mathbb{C}$ such that $Re(\omega)>0$. Then,
\begin{equation}
\label{eq:MLf}
E_{\alpha}(\omega)
= \sum_{k=0}^{\infty} \frac{\omega^{k}}{\Gamma(\alpha k +1)}.
\end{equation}

\begin{lemma}[See \cite{SidiAmmi2022}]
\label{L5}
Let $\omega \in L^{\infty}\left(0, \mathcal{T};
L^{2}(\Omega)\right) \cap H^{1}\left(0, \mathcal{T};
L^{1}(\Omega)\right)$, where $b>0$.
There exists a constant $k>0$ such that
$$
\frac{k}{E_{\alpha}\left(-\gamma \mathcal{T}^{\alpha}\right)}
\|\omega\|_{L^{\infty}\left(0, \mathcal{T}; L^{1}(\Omega)\right)}
\geq \left\|\partial_{t} \omega\right\|_{L^{1}
\left(0, \mathcal{T}; L^{1}(\Omega)\right)},
$$
where $E_{\alpha}$ is the Mittag--Leffler function \eqref{eq:MLf}.
\end{lemma}

The primary result of this section can now be formulated and proved as follows.

\begin{theorem}
\label{T4}
Model \eqref{E3.1}--\eqref{E3.3} admits at least one optimal solution $\nu^*(u^*)$
in $\left(L^\infty(\Omega_\mathcal{T})\right)^3$ that minimizes \eqref{E3.4}.
\end{theorem}

\begin{proof}
Let $\left((\nu^n, u^n)\right)_{k}$ be such that
\[
\mathcal{J}(\nu^*,u^*) = \inf\left\{\mathcal{J}(\nu, u)\right\}
= \lim_{n\rightarrow\infty} \mathcal{J}(\nu^n, u^n),
\]
where $\nu^n = (\nu_i^n)_{i=1,2,3}$ and $u^n\in U_{ad}$.
The pair $(\nu^n, u^n)$ satisfies system
\begin{equation}
\label{E5.1}
\left\{\begin{array}{l}
{ }^{\mathcal{C}} \mathcal{D}^{\alpha}_{t}\nu_1^n
= \mathcal{L}\nu_1^n + \mu\left(\nu_1^n + \nu_2^n
+ \nu_3^n\right)-\beta \nu_1^n \nu_2^n - \xi \nu_1^n - u^n \nu_1^n, \\
{ }^{\mathcal{C}} \mathcal{D}^{\alpha}_{t}\nu_2^n
= \mathcal{L}\nu_2^n + \beta \nu_1^n \nu_2^n - (\xi + \kappa) \nu_2^n,\\
{ }^{\mathcal{C}} \mathcal{D}^{\alpha}_{t}\nu_3^n
= \mathcal{L}\nu_3^n + \kappa \nu_2^n - \xi \nu_3^n + u^n \nu_1^n,
\end{array}\right.
\text{ in } \Omega_\mathcal{T},
\end{equation}
with the conditions
\begin{equation}
\label{E5.2}
\nabla \nu_i^n\cdot\vec{n} = 0,
\ \text{ on }  \partial\Omega_\mathcal{T}, \ i=1,2,3,
\end{equation}
and
\begin{equation}
\label{E5.3}
\nu_i^n(0) = \nu_i^0, \ \text{ in }  \Omega, \ i=1,2,3.
\end{equation}
Let $i\in\{1,2,3\}$. According to Theorem~\ref{T1}, the sequence $(\nu_i^n)$
is bounded in $L^\infty(0, \mathcal{T}; \mathcal{H}(\Omega))$
and in $L^p(0, \mathcal{T}; \mathcal{W}(\Omega))$.
Hence, there exists a positive constant $c$ such that
$$
\| { }^{\mathcal{C}} \mathcal{D}^{\alpha}_{t}\nu_i^n
- \mathcal{L}\nu_i^n \|_{L^2(\Omega_\mathcal{T})} \leq c.
$$
As a result, there exists a subsequence of ($\nu^n$),
still denoted as ($\nu^n$), such that
$$
\begin{aligned}
& { }^{\mathcal{C}} \mathcal{D}^{\alpha}_{t}\nu^n
- \mathcal{L}\nu^n \rightharpoonup  \phi
\text{ weakly in } \mathcal{H}(\Omega_\mathcal{T}), \\
& \nu^n \rightharpoonup  \nu^* \text { weakly in }
L^p(0, \mathcal{T}; \mathcal{W}(\Omega)).
\end{aligned}
$$
By utilizing \eqref{E4.8}, \eqref{E4.10}, and \eqref{E4.11},
we deduce that
$$
\phi = { }^{\mathcal{C}}
\mathcal{D}^{\alpha}_{t}\nu^* - \mathcal{L}\nu^*.
$$
Let us now define the space
$$
E = \left\{w\in L^p(0, \mathcal{T}; W^{1,p}_0(\Omega))
/ \partial_t w\in L^1(0, \mathcal{T}; L^1(\Omega)) \right\}.
$$
Since $W^{1,p}_0(\Omega)$ is compactly embedded in
$L^p(\Omega)\subseteq L^2(\Omega)$, it follows that
$(\nu_i^n)$ is compact in $L^2(\Omega)$. By applying
Lemma~\ref{L5}, we conclude that $(\partial_t \nu_i^n)$
is bounded in $L^1(0, \mathcal{T}; L^1(\Omega))$. Using
the classical Aubin argument \cite{Moussa2015}, we establish
the compact embedding of the space $E$ in $L^2(\Omega_\mathcal{T})$.
Therefore, there exists a subsequence of ($\nu^n$),
denoted as ($\nu^n$), such that
$$
\begin{aligned}
& \nu^n \rightharpoonup  \nu^*
\text{ weakly in } \mathcal{H}(\Omega_\mathcal{T})
\text{ and in } L^\infty\left(0, \mathcal{T}; \mathcal{H}(\Omega)\right), \\
& \nu^n \longrightarrow  \nu^*
\text{ strongly in } \mathcal{H}(\Omega_\mathcal{T}), \\
& \nu^n \longrightarrow  \nu^*
\text{ a.e. in } \mathcal{H}(\Omega_\mathcal{T}),\\
& \nu^n(T) \longrightarrow  \nu^*(T)
\text{ in } \mathcal{H}(\Omega_\mathcal{T}).
\end{aligned}
$$
Writing $\nu_1^n \nu_2^n - \nu_1^* \nu_2^*
= \left(\nu_1^n-\nu_1^*\right) \nu_2^n + \nu_1^*\left(\nu_2^n - \nu_2^*\right)$,
and by leveraging the convergence $\nu_i^n \longrightarrow \nu_i^*$ in
$L^2(\Omega_\mathcal{T})$ and the boundedness of $(\nu_1^n), (\nu_2^n)$
in $L^{\infty}(\Omega_\mathcal{T})$, we conclude that $\nu_1^n \nu_2^n
\longrightarrow \nu_1^* \nu_2^*$ in $L^2(\Omega_\mathcal{T})$.
Additionally, we have $u^n \rightarrow u^*$
in $L^2(\Omega_\mathcal{T})$ on a subsequence of
$(u^n)$ denoted again as $(u^n)$.
Utilizing the closeness and convexity of $U_{a d}$ in $L^2(\Omega_\mathcal{T})$,
we conclude that $U_{a d}$ is weakly closed. Therefore, $u^* \in U_{a d}$.
In a similar manner, $u^n \nu_1^n \longrightarrow u^* \nu_1^*$
in $L^2(\Omega_\mathcal{T})$.
Afterwards, we take the limit in the system satisfied by $\nu^n$ as
$n \rightarrow \infty$, finding that $\left(\nu^*, u^*\right)$
is an optimal solution of \eqref{E3.1}--\eqref{E3.5}.
\end{proof}

% -------------------------------------------------

\section{Necessary optimality conditions}
\label{S6}

To establish necessary optimality conditions, consider $\nu^{*}(u^{*})$
an optimal solution and $u^{\varepsilon} = u^{*} + \varepsilon u \in U_{ad}$
with $u \in U_{ad}$. Let $\nu^{\varepsilon} = (\nu_{i}^{\varepsilon})_{i=1,2,3}
= (\nu_{1}, \nu_{2}, \nu_{3})(u^{\varepsilon})$ and $\nu^{*}
= (\nu_{i}^{*})_{i=1,2,3} = (\nu_{1}, \nu_{2}, \nu_{3})(u^{*})$ be
solutions of \eqref{E3.1}--\eqref{E3.3}. By subtracting the system associated
with $\nu^{*}$ from the one corresponding to $\nu^{\varepsilon}$, and noting
that $\nu_{i}^{\varepsilon} = \nu_{i}^{*} + \varepsilon y_{i}^{\varepsilon}$
for any $i = 1, 2, 3$, we obtain that
\begin{equation}
\label{E6.1}
{ }^{\mathcal{C}} \mathcal{D}^{\alpha}_{t} y^{\varepsilon}
=  \lambda\frac{\Delta_p \nu^{\varepsilon}
-  \Delta_p \nu^{*}}{\varepsilon}
+ \frac{\Psi\left(\nu^{\varepsilon}\right)
- \Psi\left(\nu^{*}\right)}{\varepsilon}, \ (x, t)\in \Omega_\mathcal{T},
\end{equation}
with
\begin{equation}
\label{E6.2}
\nabla y_{i}^\varepsilon\cdot \vec{n}
= 0, \quad (x, t) \in \Sigma_T,
\end{equation}
and
\begin{equation}
\label{E6.3}
y_{i}^{\varepsilon}(x, 0)=0 \quad x \in \Omega.
\end{equation}
A straightforward calculation yields
$$
\frac{\Psi\left(\nu^{\varepsilon}\right)
- \Psi\left(\nu^{*}\right)}{\varepsilon}
= \mathcal{N}_{\varepsilon} y^{\varepsilon} + \mathcal{F}u,
$$
with
$$
\mathcal{N}_{\varepsilon}
=\left(
\begin{array}{ccc}
\beta-\mu \nu_{2}^{\varepsilon}- \xi -u^{\varepsilon}
& -\mu \nu_{1}^{*}+\beta & \beta \\
\mu \nu_{2}^{\varepsilon}
& \mu \nu_{1}^{*}- \xi - \kappa & 0 \\
u^{\varepsilon} & \kappa & -\xi
\end{array}\right)
\ \text{ and } \
\mathcal{F}=\left(\begin{array}{c}
-\nu_{1}^{*} \\
0 \\
\nu_{1}^{*}
\end{array}\right).
$$
On the other hand,
$$
\begin{array}{ll}
\Delta_p \nu^{\varepsilon} -  \Delta_p \nu^{*}
&= \operatorname{div}(\psi(\nabla \nu^{\varepsilon}))
- \operatorname{div}(\psi(\nabla \nu^{*}))\\
&= \psi^{\prime}(\nabla \nu^{\varepsilon})\Delta \nu^{\varepsilon}
- \psi^{\prime}(\nabla \nu^{*})\Delta \nu^{*} \\
&= \psi^{\prime}(\nabla \nu^{\varepsilon})\left(\Delta \nu^{\varepsilon}
- \Delta \nu^{*}\right) + \left(\psi^{\prime}(\nabla \nu^{\varepsilon})
- \psi^{\prime}(\nabla \nu^{*})\right)\Delta \nu^{*}.
\end{array}
$$
Then,
$$
\frac{1}{\varepsilon} \left(\Delta_p \nu^{\varepsilon} -  \Delta_p \nu^{*}\right)
= \psi^{\prime}(\nabla \nu^{\varepsilon})\frac{\Delta \nu^{\varepsilon}
- \Delta \nu^{*}}{\varepsilon} + \frac{\psi^{\prime}(\nabla \nu^{\varepsilon})
- \psi^{\prime}(\nabla \nu^{*})}{\varepsilon}\Delta \nu^{*}.
$$
Applying the mean value theorem, we can find $\theta_\varepsilon$ such that
$$
\frac{1}{\varepsilon} \left(\Delta_p \nu^{\varepsilon} - \Delta_p \nu^{*}\right)
= \psi^{\prime}(\nabla \nu^{\varepsilon})\Delta y^{\varepsilon}
+ \psi^{\prime\prime}(\theta_\varepsilon)\frac{\nabla \nu^{\varepsilon}
- \nabla \nu^{*}}{\varepsilon}\Delta \nu^{*}.
$$
Since $y_i^\varepsilon, \nu_i^\varepsilon$ are bounded in
$L^p(0, \mathcal{T};W^{1, p}_0(\Omega))$, uniformly with respect
to $\varepsilon$, $\psi$ is a regular function. Thus,
$$
\begin{array}{ll}
\lim\limits_{\varepsilon \rightarrow 0} \frac{\Delta_p \nu^{\varepsilon}
-  \Delta_p \nu^{*}}{\varepsilon}
&= \psi^{\prime}(\nabla \nu^{*})\Delta y
+ \psi^{\prime\prime}(\nabla \nu^{*})\Delta \nu^{*}\nabla y\\
&= \psi^{\prime}(\nabla \nu^{*})\nabla\cdot\nabla y
+ \operatorname{div}\left(\psi^{\prime}(\nabla \nu^{*})\right)\nabla y \\
&= \operatorname{div}\left(\psi^{\prime}(\nabla \nu^{*})\nabla y\right).
\end{array}
$$
On the other hand, the elements of the matrix $\mathcal{N}_{\varepsilon}$
are also uniformly bounded with respect to $\varepsilon$. We can now proceed
to the limit in \eqref{E6.1}--\eqref{E6.3} as $\varepsilon \rightarrow 0$.
Considering \eqref{E5.1}, we arrive at
\begin{equation}
\label{E6.4}
\left\{\begin{array}{ll}
{ }^{\mathcal{C}} \mathcal{D}^{\alpha}_{t} y
= \lambda \operatorname{div}\left(\psi^{\prime}(\nabla \nu^{*})\nabla y\right)
+ \mathcal{N}y + \mathcal{F}u,
&\text { in } \Omega_\mathcal{T} \\
\left(\nabla y_{i}\cdot \vec{n}\right)_{1\leq i\leq 3} = 0,
&\text { on } \Sigma_T,\\
y(x, 0) = 0, &\text { in } \Omega,
\end{array}\right.
\end{equation}
where
$$
\mathcal{N}
= \left(
\begin{array}{ccc}
\beta-\mu \nu_{2}^{*}- \xi -u^{*}
& -\mu \nu_{1}^{*}+\beta & \beta \\
\mu \nu_{2}^{*} & \mu \nu_{1}^{*}- \xi - \kappa & 0 \\
u^{*} & \kappa & -\xi
\end{array}\right).
$$
Applying the theory of fractional differential equations,
it follows that system \eqref{E6.4} admits a solution \cite{Khan2019,Milici2019}.
This allows us to derive the adjoint problem associated with the variable $y$.
We introduce the adjoint vector, denoted as $\rho =\left(\rho_1, \rho_2, \rho_3\right)$,
defined by
$$
\int_{0}^{\mathcal{T}}\int_{\Omega} \left({ }^{\mathcal{C}}
\mathcal{D}^{\alpha}_{t} y - \lambda \operatorname{div}\left(\psi^{\prime}(\nabla \nu^{*})
\nabla y\right)\right) \rho dx dt = \int_{0}^{\mathcal{T}}
\int_{\Omega} \left(\mathcal{N}y + \mathcal{F}u\right) \rho dx dt.
$$
Because of
$$
\begin{array}{ll}
\int_{\Omega}\operatorname{div}\left(\psi^{\prime}(\nabla \nu^{*})\nabla y\right) \rho dx
&=  -\int_{\Omega} \psi^{\prime}(\nabla \nu^{*})\nabla y\cdot\nabla \rho dx
+ \int_{\partial\Omega} \rho\psi^{\prime}(\nabla \nu^{*})
\nabla y\cdot\vec{n} d\sigma\\
&= -\int_{\Omega} \psi^{\prime}(\nabla \nu^{*})\nabla y\cdot\nabla \rho dx\\
&=  \int_{\Omega} y \operatorname{div}\left( \psi^{\prime}(\nabla \nu^{*})\nabla \rho\right) dx
- \int_{\partial\Omega} y\psi^{\prime}(\nabla \nu^{*})\nabla \rho\cdot\vec{n} d\sigma\\
&=  \int_{\Omega} y \operatorname{div}\left( \psi^{\prime}(\nabla \nu^{*})\nabla \rho\right) dx,
\end{array}
$$
and Lemma~\ref{L4}, we have
$$
\int_{0}^{\mathcal{T}}({ }^{\mathcal{C}} \mathcal{D}^{\alpha}_{t} y) \ \rho d t
= \int_{0}^{\mathcal{T}} y \ ({ }^{\mathcal{C}}_{\mathcal{T}}
\mathcal{D}^{\alpha}_{t} \rho)  d t + \frac{1}{\Gamma(1-\alpha)}
\rho(x, \mathcal{T}) \int_{0}^{\mathcal{T}} y(t) (\mathcal{T}-t)^{-\alpha} d t.
$$
As a result, the dual system, corresponding to \eqref{E3.1}--\eqref{E3.3},
can be expressed as
\begin{equation}
\label{E6.5}
\left\{\begin{array}{ll}\vspace{0.15cm}
{ }^{\mathcal{C}}_{\mathcal{T}} \mathcal{D}^{\alpha}_{t} \rho
- \lambda \operatorname{div}\left(\psi^{\prime}(\nabla \nu^{*})\nabla \rho\right)
- \mathcal{N}\rho = \mathcal{M}^{*} \mathcal{M} \nu^{*},
&\text { in } \Omega_\mathcal{T} \\ \vspace{0.15cm}
\nabla \rho_{i}\cdot \vec{n} = 0, \ i=1,2,3, &\text { on } \Sigma_T,\\
\rho(x, \mathcal{T}) = \mathcal{M}^{*} \mathcal{M}
\nu^{*}(x, \mathcal{T}), &\text { in } \Omega,
\end{array}\right.
\end{equation}
with $\mathcal{M}$ the matrix defined by
$$
\mathcal{M} = \left(\begin{array}{lll}
0 & 0 & 0 \\
0 & 1 & 0 \\
0 & 0 & 0
\end{array}\right).
$$
By \cite{Khan2019,Milici2019}, system \eqref{E6.5} has a solution.

The following theorem establishes necessary optimality conditions for $u^*$.

\begin{theorem}
\label{T5}
If $u^{*}$ is a solution of our optimal control problem, then
\begin{equation}
\label{E6.6}
\begin{aligned}
u^{*} &=\min \left\{1, \max \left(-\frac{1}{\eta}
\mathcal{F}^{*} \rho, 0\right)\right\} \\
&=\min \left\{1, \max \left(0, -\frac{\nu_{1}^{*}}{\eta}
\left(\rho_{1}-\rho_{3}\right)\right)\right\},
\end{aligned}
\end{equation}
where $(\nu^*, u^*)$ is solution of \eqref{E3.1}--\eqref{E3.5} and
$\rho$ is a solution of \eqref{E6.5}.
\end{theorem}

\begin{proof}
Let us assume $(\nu^*, u^*)$ to be an optimal solution
to \eqref{E3.1}--\eqref{E3.5}. This allows us to state the following:
$$
\begin{aligned}
\mathcal{J}(\nu^*, u^*)
&=\left\|\nu_{2}^*\right\|_{L^{2}(\Omega_\mathcal{T})}^{2}+\left\|\nu_{2}^*(\cdot,
\mathcal{T})\right\|_{L^{2}(\Omega)}^{2}+\eta\|u^*\|_{L^{2}(\Omega_\mathcal{T})}^{2} \\
&=\int_{0}^{\mathcal{T}}\|\mathcal{M} \nu^*\|_{\mathcal{H}(\Omega)}^{2} d t
+\|\mathcal{M} \nu^*(\cdot, \mathcal{T})\|_{\mathcal{H}(\Omega)}^{2}
+\eta\|u^*\|_{L^{2}(\Omega_\mathcal{T})}^{2}.
\end{aligned}
$$
For a sufficiently small $\varepsilon>0$, considering that the minimum
of the objective functional is achieved at $u^{*}$, the inequality
$$
\frac{\mathcal{J}\left(u^{*}+\varepsilon w\right)
	- \mathcal{J}(u)}{\varepsilon} \geq 0
$$
is equivalent to
$$
\int_{0}^{\mathcal{T}}\left(F^{*} \rho
+ \eta u^{*}, w\right)_{L^{2}(\Omega)} d t \geq 0, \ \forall w \in U_{a d}.
$$
Since $\mathcal{H}(\Omega)$ and $L^{2}(\Omega_\mathcal{T})$
are Hilbert spaces, then
\begin{align*}
&\mathcal{J}^{\prime}\left(\nu^{*}, u^{*}\right)(w)
\\=& \lim _{\varepsilon \rightarrow 0} \frac{1}{\varepsilon}
\left(\mathcal{J}\left(\nu^{\varepsilon}, u^{\varepsilon}\right)
- \mathcal{J}\left(\nu^{*}, u^{*}\right)\right) \\
=& \lim _{\varepsilon \rightarrow 0} \frac{1}{\varepsilon}
\left(\int_{0}^{\mathcal{T}} \int_{\Omega} \left(\left(\nu_{2}^{\varepsilon}\right)^{2}
-\left(\nu_{2}^{*}\right)^{2}\right) d x d t
+ \int_{\Omega} \left(\left(\nu_{2}^{\varepsilon}(x, \mathcal{T})\right)^{2}
-\left(\nu_{2}^{*}(x, \mathcal{T})\right)^{2}\right) d x\right. \\
&\qquad +\left.\eta \int_{0}^{\mathcal{T}} \int_{\Omega}\left(
\left(u^{\varepsilon}\right)^{2}-\left(u^{*}\right)^{2}\right) d x d t\right) \\
=& \lim _{\varepsilon \rightarrow 0}\left(\int_{0}^{\mathcal{T}}
\int_{\Omega}\left(\frac{\nu_{2}^{\varepsilon}-\nu_{2}^{*}}{\varepsilon}\right)
\left(\nu_{2}^{\varepsilon}+\nu_{2}^{*}\right) d x d t
+ \int_{\Omega}\left(\frac{\nu_{2}^{\varepsilon}-\nu_{2}^{*}}{\varepsilon}\right)
\left(\nu_{2}^{\varepsilon}+\nu_{2}^{*}\right)(x, \mathcal{T}) d x\right. \\
&\qquad +\left.\eta \int_{0}^{\mathcal{T}} \int_{\Omega}(\varepsilon
w^{2}+2 w u^{*}) d x d t \right).
\end{align*}
Because $\nu_{2}^{\varepsilon} \rightarrow \nu_{2}^{*}$ belongs to
$L^{2}(\Omega_\mathcal{T})$ and $\nu_{2}^{\varepsilon},
\nu_{2}^{*} \in L^{\infty}(\Omega_\mathcal{T})$, then
$$
\begin{aligned}
\mathcal{J}^{\prime}\left(\nu^{*}, u^{*}\right)(w)
&= 2 \int_{0}^{\mathcal{T}} \int_{\Omega}\left(\nu_{2}^{*}\right)
\nu^{\prime}\left(u^{*}\right) w \ d x d t\\
&\quad + 2 \int_{\Omega}\left(\left(\nu_{2}^{*}\right) \nu^{\prime}
\left(u^{*}\right) w\right)(x, \mathcal{T}) d x
+ 2 \eta \int_{0}^{\mathcal{T}} \int_{\Omega} w u^{*} d x d t.
\end{aligned}
$$
This is the same as
$$
\begin{aligned}
J^{\prime}\left(\nu^{*}, u^{*}\right)(w)
&= 2 \int_{0}^{\mathcal{T}}\left( \mathcal{M} \nu^{*},
\mathcal{M} y\right)_{\mathcal{H}(\Omega)} d t \\
&\quad  + 2\left( \mathcal{M} \nu^{*}(x, \mathcal{T}),
\mathcal{M} y(x, \mathcal{T})\right)_{\mathcal{H}(\Omega)}
+ 2 \eta \int_{0}^{\mathcal{T}}\left( u^{*}, w\right)_{L^{2}(\Omega)} d t,
\end{aligned}
$$
where $y = \nu^{\prime}\left(u^{*}\right) w$ represents the unique
solution of \eqref{E6.4} related to $w$. By using \eqref{E6.4}
and \eqref{E6.5}, we get
$$
\begin{aligned}
\int_{0}^{\mathcal{T}}\left( \mathcal{M} \nu^{*}, \mathcal{M} y \right)_{\mathcal{H}(\Omega)}
&d t +\left( \mathcal{M} \nu^{*}(x, \mathcal{T}),
\mathcal{M} y(x, \mathcal{T})\right)_{\mathcal{H}(\Omega)} \\
&=\int_{0}^{\mathcal{T}}\left( \mathcal{M}^{*} \mathcal{M} \nu^{*},
y\right)_{\mathcal{H}(\Omega)} d t + \left( \mathcal{M}^{*}
\mathcal{M} \nu^{*}(x, \mathcal{T}), y(x, \mathcal{T})\right)_{\mathcal{H}(\Omega)}\\
&= \int_{0}^{\mathcal{T}}\left({ }^{\mathcal{C}}_{\mathcal{T}}
\mathcal{D}^{\alpha}_{t} \rho - \lambda \operatorname{div}\left(\psi^{\prime}
(\nabla \nu^{*})\nabla \rho\right)
- \mathcal{N}\rho \ , \ y\right)_{\mathcal{H}(\Omega)} d t\\
&\quad +\left( \mathcal{M}^{*} \mathcal{M} \nu^{*}(x, \mathcal{T}),
y(x, \mathcal{T})\right)_{\mathcal{H}(\Omega)}\\
&= \int_{0}^{\mathcal{T}}\left( \rho \ , \ { }^{\mathcal{C}}
\mathcal{D}^{\alpha}_{t} y - \lambda \operatorname{div}\left(\psi^{\prime}(
\nabla \nu^{*})\nabla y\right) - \mathcal{N}y\right)_{\mathcal{H}(\Omega)} d t,
\end{aligned}
$$
that is,
$$
\begin{aligned}
\int_{0}^{\mathcal{T}}\left( \mathcal{M} \nu^{*}, \mathcal{M} y \right)_{\mathcal{H}(\Omega)}
&d t +\left( \mathcal{M} \nu^{*}(x, \mathcal{T}),
\mathcal{M} y(x, \mathcal{T})\right)_{\mathcal{H}(\Omega)} \\
&= \int_{0}^{\mathcal{T}}\left( \rho,
\mathcal{F}w)\right)_{\mathcal{H}(\Omega)} d t \\
&= \int_{0}^{\mathcal{T}}\left( \mathcal{F}^{*} \rho,
w\right)_{L^{2}(\Omega)} d t,
\end{aligned}
$$
and
$$
\mathcal{J}^{\prime}\left(\nu^{*}, u^{*}\right)(w)
= 2\int_{0}^{\mathcal{T}}\left(\mathcal{F}^{*} \rho
+ \eta u^{*}, w\right)_{L^{2}(\Omega)} d t.
$$
With the G\^{a}teaux differentiability of $\mathcal{J}$ at $u^*$
and the convexity of $U_{ad}$, we can affirm that
$$
0\leq \mathcal{J}^{\prime}\left(\nu^{*}, u^{*}\right)(\phi-u^*)
$$
holds for all $\phi\in U_{ad}$. Therefore, we can state that
\begin{align*}
u^{*}
&=\min \left\{1, \max \left(-\frac{1}{\eta} \mathcal{F}^{*} \rho, 0\right)\right\}\\
&= \min \left\{1, \max \left(0, -\frac{\nu_{1}^{*}}{\eta}
\left(\rho_{1}-\rho_{3}\right)\right)\right\}.
\end{align*}
The proof is complete.
\end{proof}

While our study provides a thorough analysis of the fractional spatiotemporal SIR model, 
including the derivation of necessary optimality conditions for the associated 
control problem, the derivation of sufficient conditions for optimality remains 
an open question. Sufficient conditions often require additional structural properties 
of the objective functional or state dynamics. For example, strict convexity 
of the objective functional
\[
\mathcal{J}(u) = \|I(\cdot, \mathcal{T})\|_{L^2(\Omega)}^2 
+ \|I\|_{L^2(\Omega_\mathcal{T})}^2 + \eta \|u\|_{L^2(\Omega_\mathcal{T})}^2,
\]
requires verifying that
\[
\frac{\partial^2 \mathcal{J}}{\partial u^2} > 0 
\quad \text{for all admissible } u \in U_{\text{ad}}.
\]

Similarly, monotonicity properties of the state operator \(\mathcal{A}[S, I, R, u]\) defined by
\[
\mathcal{A}[S, I, R, u] =
\begin{bmatrix}
{ }^{\mathcal{C}} \mathcal{D}^\alpha_t S - \lambda_1 \Delta_p S - \beta N + \mu S I + \xi S + uS \\
{ }^{\mathcal{C}} \mathcal{D}^\alpha_t I - \lambda_2 \Delta_p I - \mu S I + (\xi + \kappa) I \\
{ }^{\mathcal{C}} \mathcal{D}^\alpha_t R - \lambda_3 \Delta_p R - \kappa I + \xi R - uS
\end{bmatrix},
\]
governing the dynamics, such as
\[
\langle \mathcal{A}[S_1, I_1, R_1, u] - \mathcal{A}[S_2, I_2, R_2, u], 
S_1 - S_2 \rangle \geq 0,
\]
might provide additional insights into sufficient conditions for optimality.

However, these properties are not immediately evident in our current framework 
involving fractional derivatives and the \(p\)-Laplacian operator. Addressing 
this limitation would include exploring the functional properties of the system 
in greater depth and potentially imposing additional constraints or regularity 
assumptions. Future research could focus on establishing such conditions, providing 
stronger guarantees for the optimality of the control program.

% -------------------------------------------------

\section{Numerical results}
\label{S7}

The primary objective of this section is to present approximate
and numerical results to our fractional model. Therefore, whether
in the absence of a vaccination program or in its presence, we study
the impact of the $\alpha$-order derivation and the coefficient
of diffusion $p$ on the spatial spread of infection over 80 days.

In the following, we presume that the illness begins in the sub-domain
$\Omega_0 = cell(11, 11)$ when the disease is born in the middle
of $\Omega$, where $\Omega$ symbolizes an urban area inhabited
by the individuals under examination, covering a
$21km \times 21km$ square grid of spatial extent.
Except for the sub-domain $\Omega_0$ of $\Omega$, we assume that all
susceptible individuals are homogeneously and uniformly distributed
at $t = 1$ with 50 in each $1km \times 1km$ cell, where we introduce
10 infected individuals in $\Omega_0$ while maintaining 40 susceptible
individuals in the same area.

Due to the unavailability of real-world data in this study, we have selected 
parameter values based on previous works 
\cite{Alshomrani2021, Ding2022, SidiAmmi2022, SidiAmmi2023, Zinihi2024MM}, 
which qualitatively represent the dynamics of disease spread. These values are 
chosen to capture the general progression of an epidemic, while respecting 
realistic assumptions about the relationships between the parameters, as observed 
in real-world scenarios. For example, certain rates, such as transmission or recovery 
rates, are expected to be larger or smaller depending on biological or 
intervention-related processes. The parameters in the SIR model reflect these 
various processes, and their interrelationships offer valuable insights into 
the dynamics of disease spread. A summary of the parameter values and initial 
conditions used in our simulations can be found in Table~\ref{Tab2}.

\begin{table}[H]
\centering
\caption{Initial conditions and parameters values.}\label{Tab2}
{\footnotesize \begin{tabular}{|c|c|c|c|}
\hline Symbol & Description & Unit & Value \\
\hline\hline $S_{0}(x,y)$ & Initial susceptible individuals
& $people\cdot km^{-2}$ & $50$ for $(x,y)\notin\Omega_0$\\
& & & $40$ for $(x,y)\in\Omega_0$ \\
\hline $I_{0}(x,y)$ & Initial infected individuals
& $people\cdot km^{-2}$ & $0$ for $(x,y)\notin\Omega_0$\\
& & & $10$ for $(x,y)\in\Omega_0$ \\
\hline $R_{0}(x,y)$ & Initial recovered individuals
& $people\cdot km^{-2}$ & $0$ for $(x,y)\in\Omega$ \\
\hline $\lambda_{1}=\lambda_{2}=\lambda_{3}$
& Diffusion coefficient & $km^{2}\cdot day^{-1}$ & $0.1$ \\
\hline $\beta$ & Birth rate & $day^{-1}$ & $0.02$ \\
\hline $\xi$ & Natural death rate & $day^{-1}$ & $0.03$ \\
\hline $\mu$ & Transmission rate & $(people\cdot day)^{-1}km^{2}$ & $0.9$ \\
\hline $\kappa$ & Recovery rate & $day^{-1}$ & $0.04$ \\
\hline $\mathcal{T}$ & Final time & $day$ & $80$ \\ \hline
\end{tabular}}
\end{table}

% --------------------------------------------------------------------

\subsection{Forward-backward sweep method algorithm}
\label{S7.1}

The approach we apply is founded on an iterative discrete scheme
similar to the forward-backward sweep method. To resolve our fractional
system \eqref{E3.1}--\eqref{E3.5}, we have used MatLab.
Systems \eqref{E3.1} and \eqref{E6.5} are numerically integrated by
using an explicit finite difference method to approximate the left-right
$\mathcal{C}$-fractional derivatives. By applying a forward approach in time,
we solve the state problem, and a backward approach in time solves the dual
system according to the transversality conditions. The preceding discussion
can be succinctly summarized in the algorithm organigram depicted in Figure~\ref{F2}.

% --------------------------------------------------------------------

\subsection{Numerical findings without a vaccination program}
\label{S7.2}

Our numerical results, obtained by varying the values of $p$ and $\alpha$,
without implementing a vaccination strategy, are depicted in Figures~\ref{F3},
\ref{F4}, and \ref{F5}. These figures illustrate the progression of the epidemic
from the first day to the last, revealing a consistent pattern of susceptible
citizens spreading from the center of $\Omega$ to its periphery, with infected
citizens multiplying rapidly. In Figure~\ref{F3}, for the case of $\alpha = 1$
and $p \in \{15, 10, 5\}$, we observe that on the final day, there are approximately
50 infected individuals per cell. This implies that the epidemic takes 80 days
to fully cover all $\Omega$. Contrary, as depicted in Figures~\ref{F4} and \ref{F5},
it takes more than 80 days for the pandemic to spread and encompass the entire
region when $\alpha \in \{0.95, 0.9\}$ and $p \in \{15, 10, 5\}$. In these scenarios,
there are approximately 43 infected individuals per cell.

% --------------------------------------------------------------------

\subsection{Numerical findings with a vaccination program}
\label{S7.3}

Now, the vaccination program is initiated from the first day of the simulation.
In this scenario, we assess the evolution of infection prevalence over an
80-days period, during which susceptible individuals move to the category of recovered.
Maintaining the same parameter configurations as presented in Section~\ref{S7.2},
Figure~\ref{F6} portrays the distribution of 30 infected citizens per grid cell
and 17 recovered individuals per cell in the center of $\Omega$. In contrast,
Figures~\ref{F7} and \ref{F8} illustrate an environment where the central region of
$\Omega$ contains approximately 5 susceptible persons, 33 infected individuals,
and 12 recovered per cell. For a comprehensive visualization of the epidemic
dynamics, under various parameter settings, $p \in \{15, 10, 5\}$
and $\alpha \in \{1, 0.95, 0.9\}$, Figure~\ref{F9} provides videos,
showcasing the spread of the epidemic over a 80-days period during
the implementation of the vaccination strategy.

% -------------------------------------------------
\begin{sidewaysfigure}
% ======================================
\medskip
\centering
\scalebox{0.89}{ \begin{tikzpicture}[node distance=2cm]
\node (St) [rectangle, rounded corners, minimum width=1.7cm,
minimum height=1cm,text centered, draw=black, fill=red!30] {Start};
%============
\node (Init) [trapezium, trapezium stretches=true, trapezium left angle=70,
trapezium right angle=110, minimum width=2.5cm, minimum height=1cm,
text centered, draw=black, fill=blue!30, right of=St, xshift=1.5cm] {Initialization};
%============
\node (ErTe) [diamond, minimum width=1cm, minimum height=0.5cm, text centered,
draw=black, fill=green!30, right of=Init, xshift=1.5cm]
{$Err_{test} { }^{{ }^{\color{red} a}}<0$};
%============
\node (A) [rectangle, minimum width=2.3cm, minimum height=1cm, text centered,
text width=3cm, draw=black, fill=orange!30, right of=ErTe, xshift=2.5cm, yshift=-1cm]
{Solve the state system SIR};
%============
\node (B) [rectangle, minimum width=2.5cm, minimum height=1cm, text centered,
text width=3cm, draw=black, fill=orange!30, right of=A, xshift=2.5cm]
{Solve the adjoint problem};
%============
\node (C) [rectangle, minimum width=2.5cm, minimum height=1cm, text centered,
text width=3cm, draw=black, fill=orange!30, above of=B, yshift=0.7cm]
{Update the optimal control$^{\color{red} b}$};
%============
\node (D) [rectangle, minimum width=2.5cm, minimum height=1cm, text centered,
text width=3cm, draw=black, fill=orange!30, left of=C, xshift=-2.5cm]
{Update the condition loop$^{\color{red} c}$};
%============
\node (FinSol) [trapezium, trapezium stretches=true, trapezium left angle=70,
trapezium right angle=110, minimum width=2cm, minimum height=1cm,
text centered, draw=black, fill=blue!30, below of=Init] {Final Solution};
%============
\node (End) [rectangle, rounded corners, minimum width=1.5cm,
minimum height=1cm,text centered, draw=black,
fill=red!30, left of=FinSol, xshift=-1.5cm] {End};
%============ %============ %============ %============
\draw [thick,->,>=stealth] (St) -- ++ (Init);
\draw [thick,->,>=stealth] (Init) -- ++ (ErTe);
\draw [thick,->,>=stealth] (ErTe.east) |- (A) node[near end,left,xshift=0.2cm,yshift=-0.3cm]{Yes};
\draw [thick,->,>=stealth] (A) -- ++ (B);
\draw [thick,->,>=stealth] (B) -- ++ (C);
\draw [thick,->,>=stealth] (C) -- ++ (D);
\draw [thick,->,>=stealth] (D) -| (ErTe);
\draw [thick,->,>=stealth] (ErTe) |- (FinSol) node[midway,above,xshift=-0.5cm,yshift=0.05cm]{No};
\draw [thick,->,>=stealth] (FinSol) -- ++ (End);
\end{tikzpicture}}
\caption{Algorithm organigram for the proposed SIR epidemic model.}\label{F2}
% ======================================
\vspace{0.5cm}

% ======================================
\begin{subfigure}[t]{0.32\textwidth}
\includegraphics[scale=0.27, height=2.8cm]{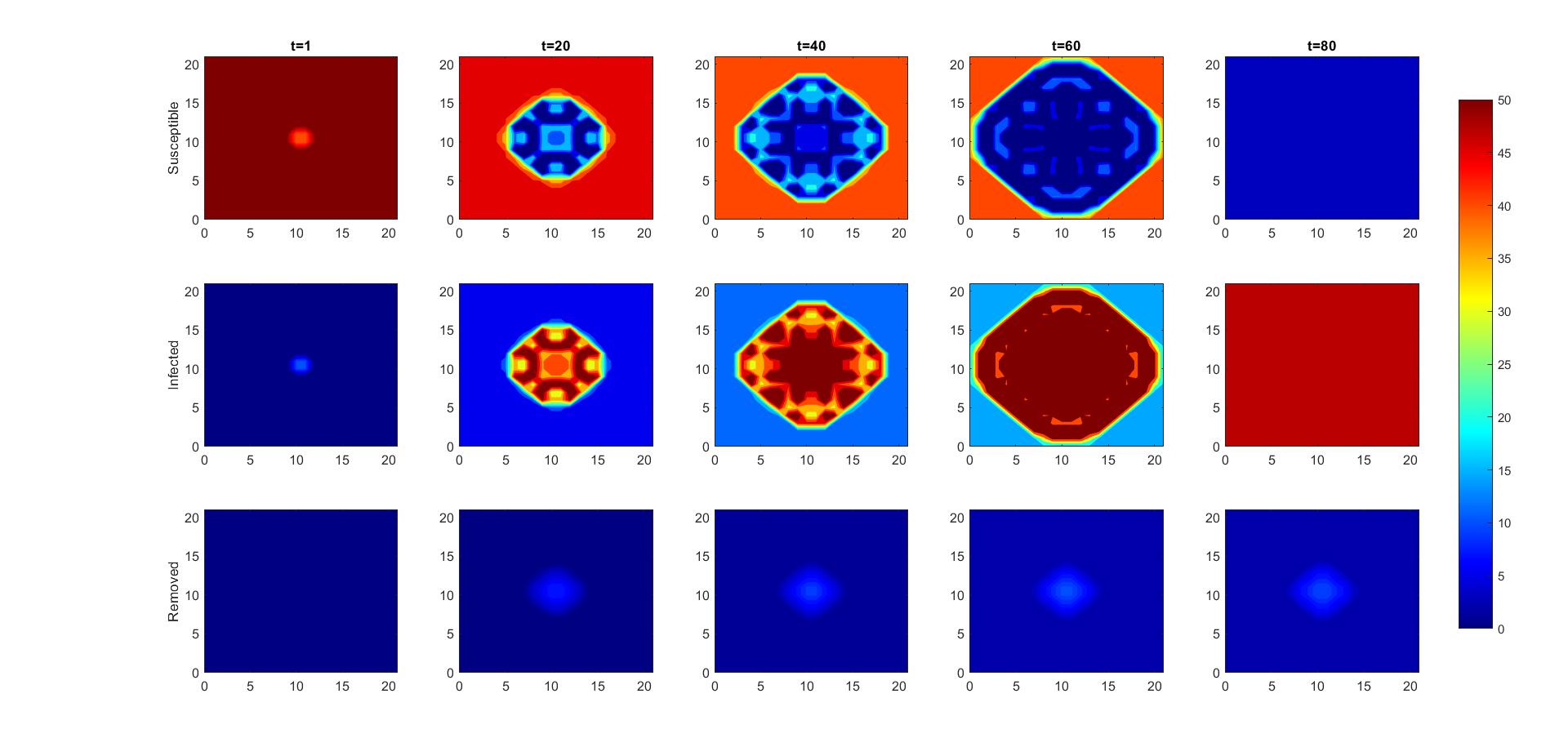}
\caption{For $p=15$}\label{F3A}
\end{subfigure}
\hfill
\begin{subfigure}[t]{0.32\textwidth}
\includegraphics[scale=0.27, height=2.8cm]{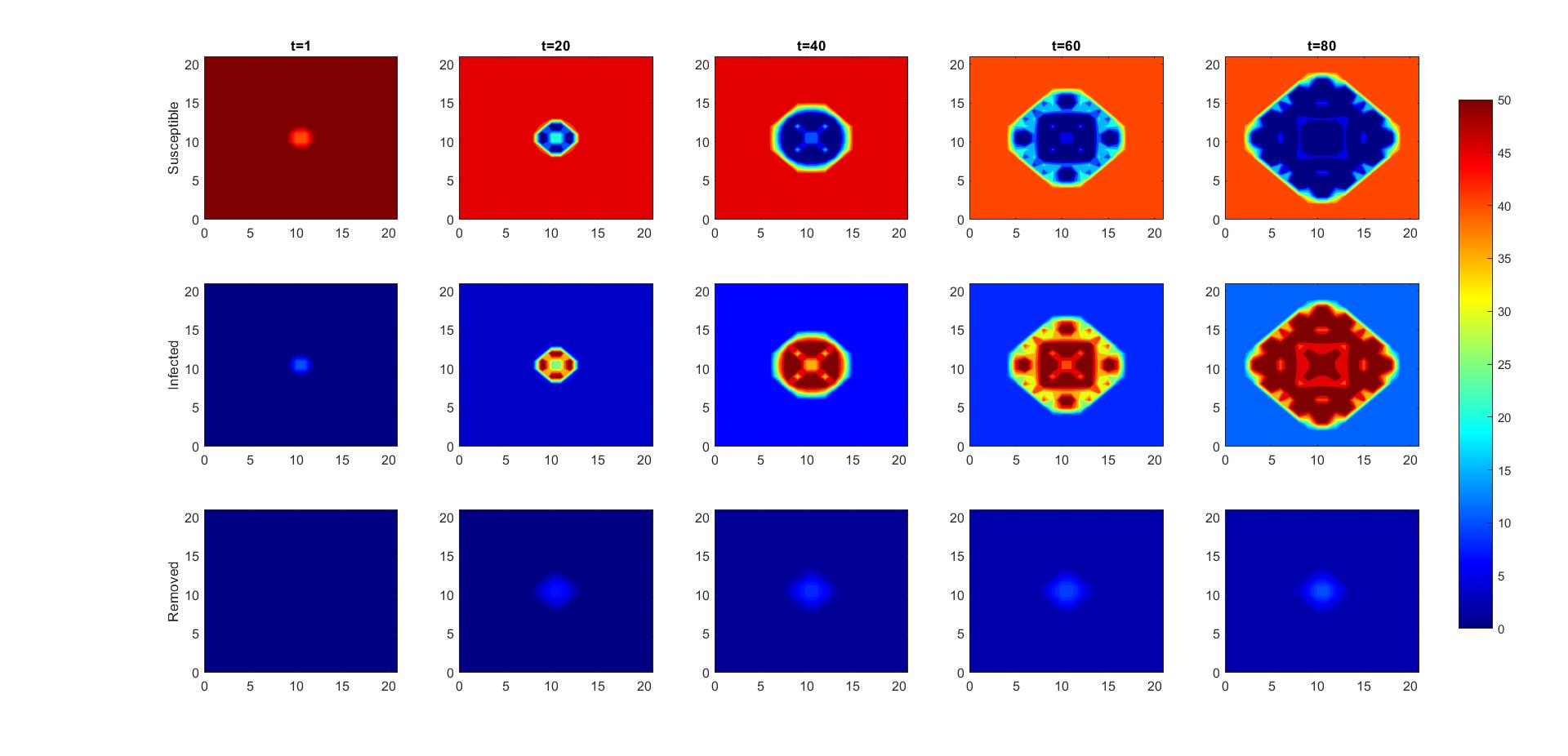}
\caption{For $p=10$}\label{F3B}
\end{subfigure}
\hfill
\begin{subfigure}[t]{0.32\textwidth}
\includegraphics[scale=0.27, height=2.8cm]{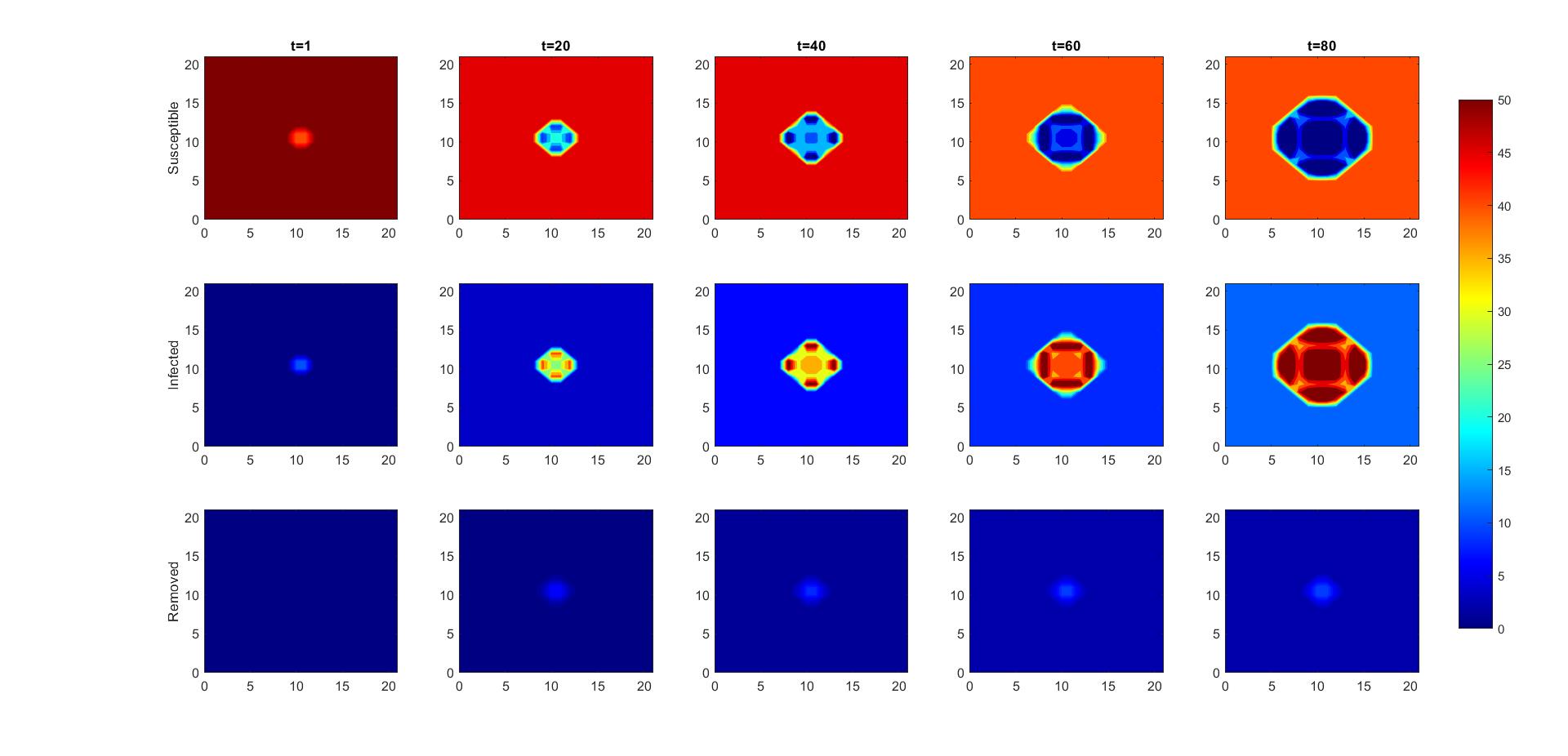}
\caption{For $p=5$}\label{F3C}
\end{subfigure}
\caption{Numerical findings of \eqref{E3.1}--\eqref{E3.3}
without vaccination program for $\alpha=1$.}\label{F3}
% ======================================
\begin{flushleft}
{\small
\noindent\rule{4cm}{0.4pt}\\
\quad Let $S = \{x_i = 1 + i\delta_x / i=0,\cdots N_x - 1\}$ be a uniform subdivision 
on $[1, \mathcal{T}]$ where $N_x$ and $\delta_x$ are the numbers and the step size of 
this subdivision, respectively. The number of steps in time is $N_t$, with the step 
time denoted by $\delta_t$; $\delta$ is the tolerance, 
and its value is $10^{-3}$ and $Err_{test} = -1$.\\
\quad ${}^{a}$Condition of the ``\textit{while}" loop.\\
\quad ${}^{b}$Using equation \eqref{E6.6}, update the optimal control $u$.\\
\quad ${}^{c}$$Err_{test} = \min\left\{\psi_{1}, \psi_{2}, \psi_{2}, \psi_{4}, 
\psi_{5}, \psi_{6}, \psi_{7} \right\}$, where $\psi_{1} = \delta\|S\| 
- \|S-S_{old}\|$, $\psi_{2} = \delta\|I\| - \|I-I_{old}$, 
$\psi_{3} = \delta\|R\| - \|R-R_{old}\|$, 
$\psi_{4} = \delta\|\rho_{1}\| - \|\rho_{1}-\rho_{old1}\|$, 
$\psi_{5} = \delta\|\rho_{2}\| - \|\rho_{2}-\rho_{old2}\|$, 
$\psi_{6} = \delta\|\rho_{3}\| - \|\rho_{3}-\rho_{old3}\|$, 
and $\psi_{7} = \delta\|u\| - \|u-u_{old}\|$.
}
\end{flushleft}
\end{sidewaysfigure}
\begin{sidewaysfigure}
% ======================================
\bigskip
\centering
\begin{subfigure}[t]{0.32\textwidth}
\includegraphics[scale=0.3, height=2.7cm]{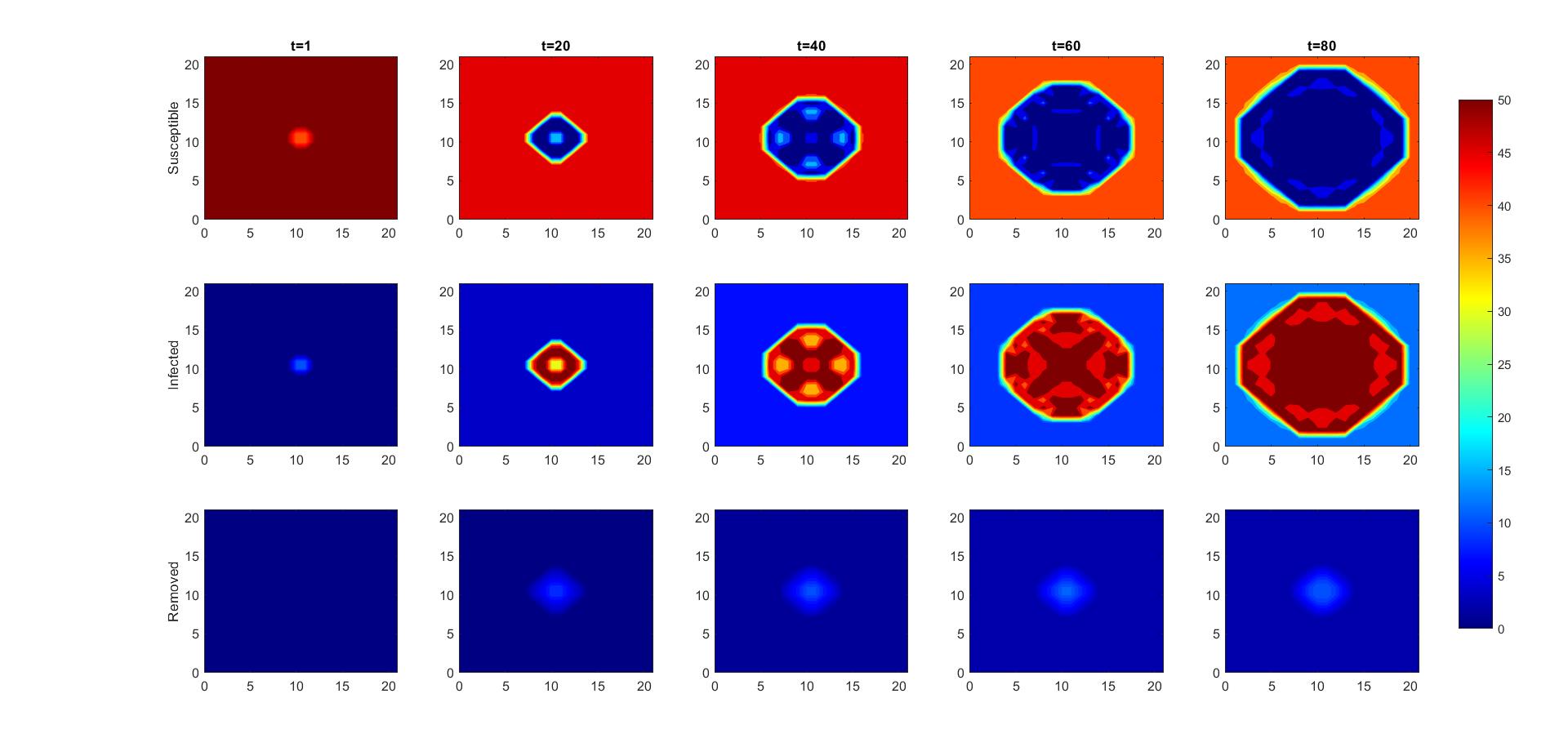}
\caption{For $p=15$}\label{F4A}
\end{subfigure}
\hfill
\begin{subfigure}[t]{0.32\textwidth}
\includegraphics[scale=0.3, height=2.7cm]{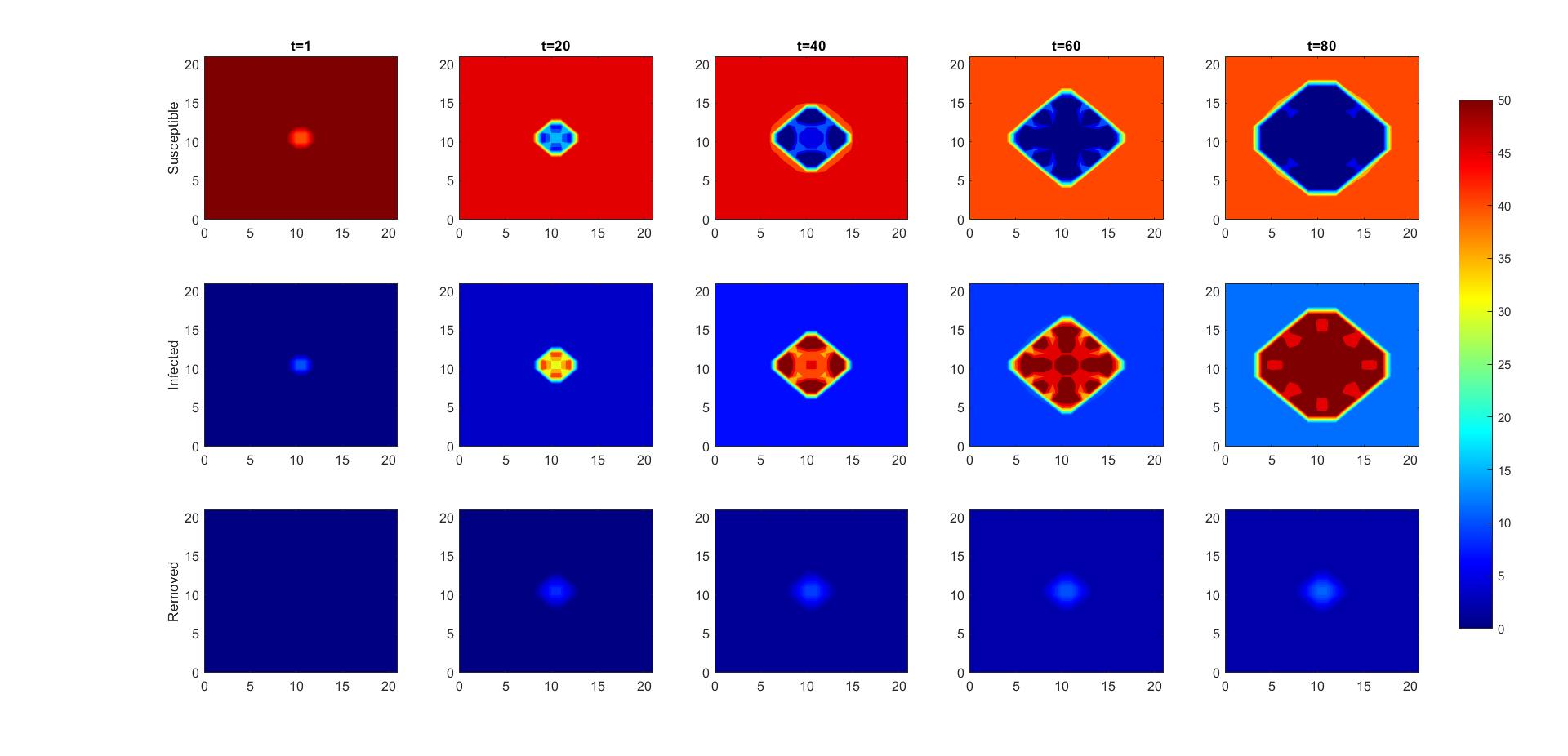}
\caption{For $p=10$}\label{F4B}
\end{subfigure}
\hfill
\begin{subfigure}[t]{0.32\textwidth}
\includegraphics[scale=0.3, height=2.7cm]{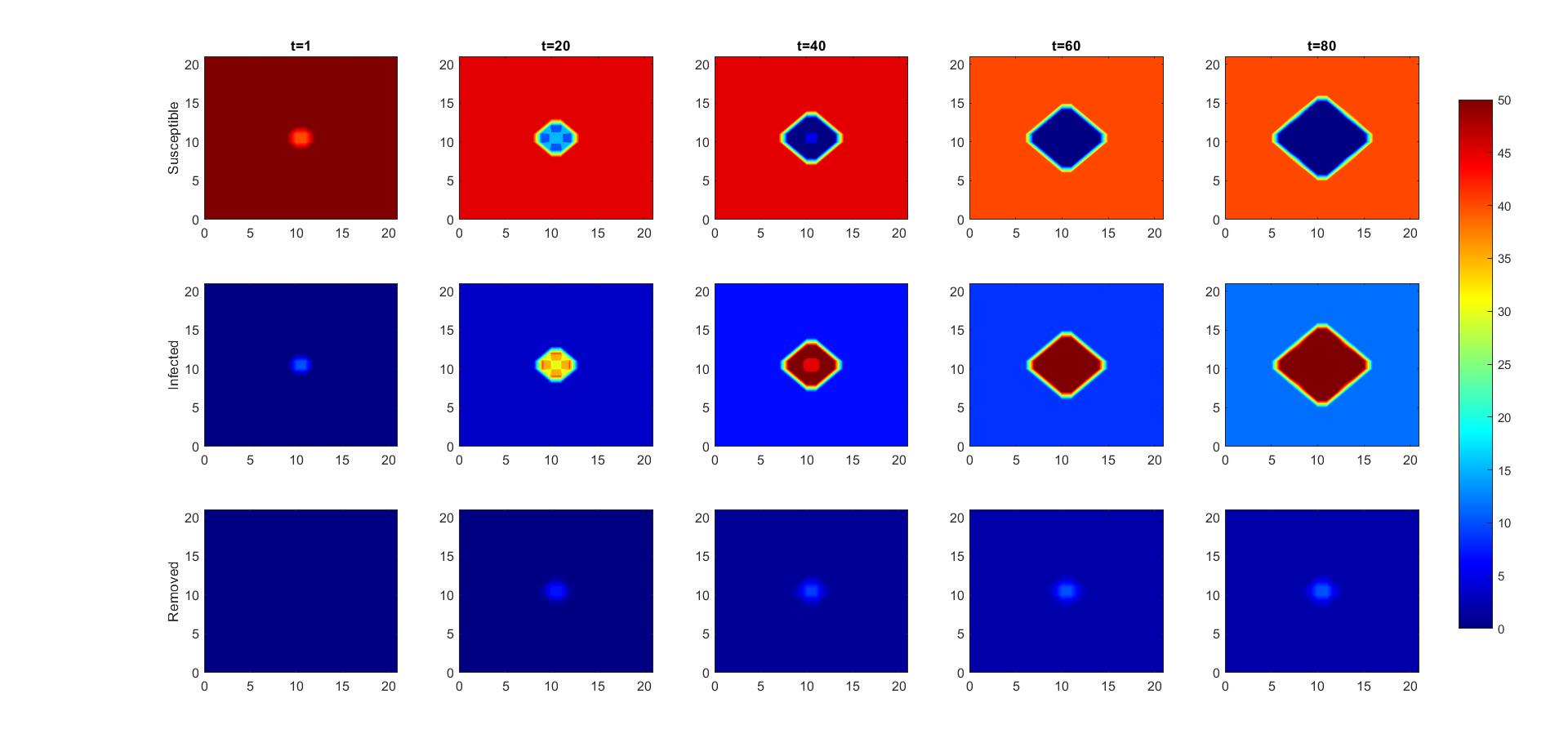}
\caption{For $p=5$}\label{F4C}
\end{subfigure}
\caption{Numerical findings of \eqref{E3.1}--\eqref{E3.3}
without vaccination program for $\alpha=0.95$.}\label{F4}
% ======================================
\vspace{0.2cm}
% ======================================
\begin{subfigure}[t]{0.32\textwidth}
\includegraphics[scale=0.3, height=2.7cm]{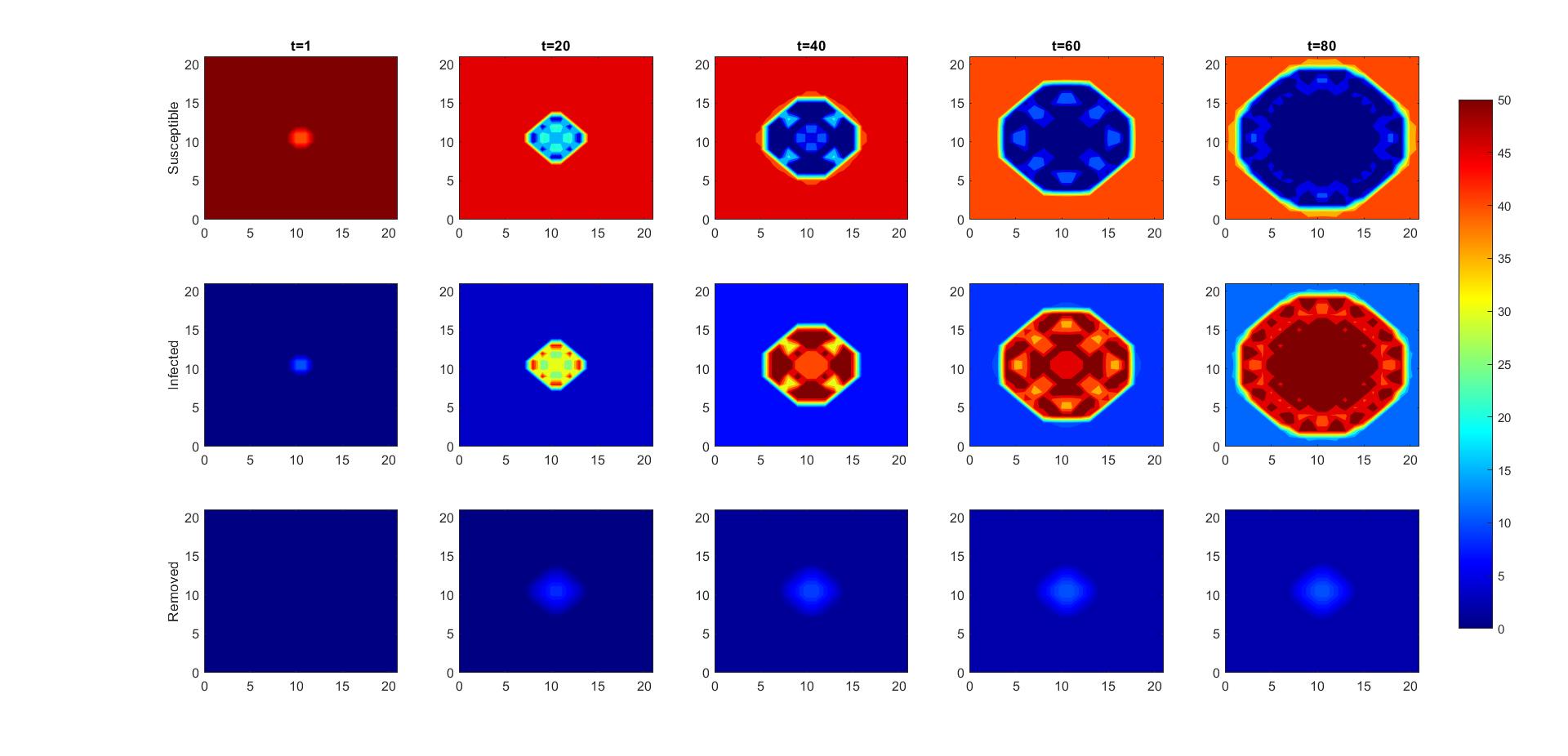}
\caption{For $p=15$}\label{F5A}
\end{subfigure}
\hfill
\begin{subfigure}[t]{0.32\textwidth}
\includegraphics[scale=0.3, height=2.7cm]{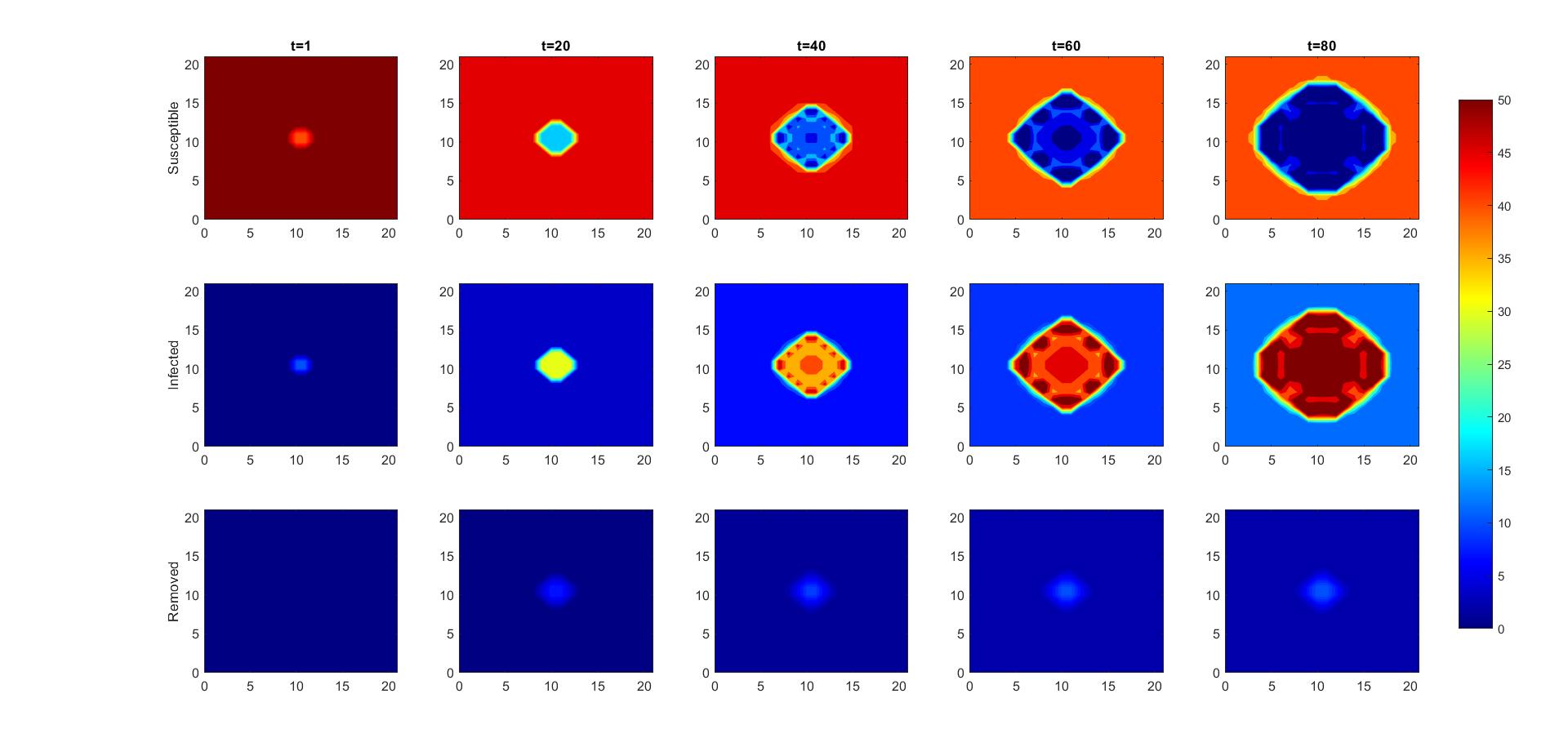}
\caption{For $p=10$}\label{F5B}
\end{subfigure}
\hfill
\begin{subfigure}[t]{0.32\textwidth}
\includegraphics[scale=0.3, height=2.7cm]{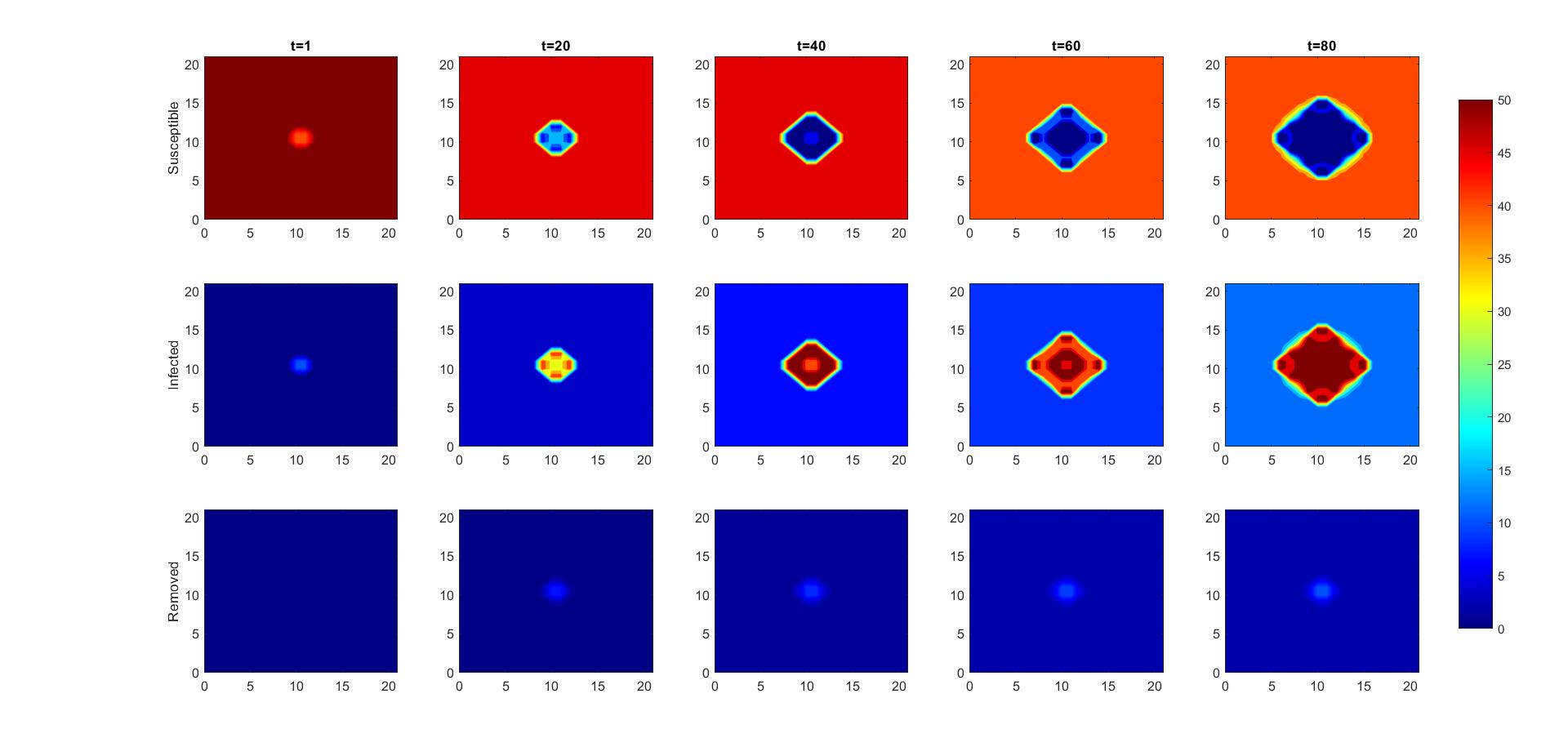}
\caption{For $p=5$}\label{F5C}
\end{subfigure}
\caption{Numerical findings of \eqref{E3.1}--\eqref{E3.3}
without vaccination program for $\alpha=0.9$.}\label{F5}
% ======================================

\vspace{0.2cm}
% ======================================
\begin{subfigure}[t]{0.32\textwidth}
\includegraphics[scale=0.3, height=2.7cm]{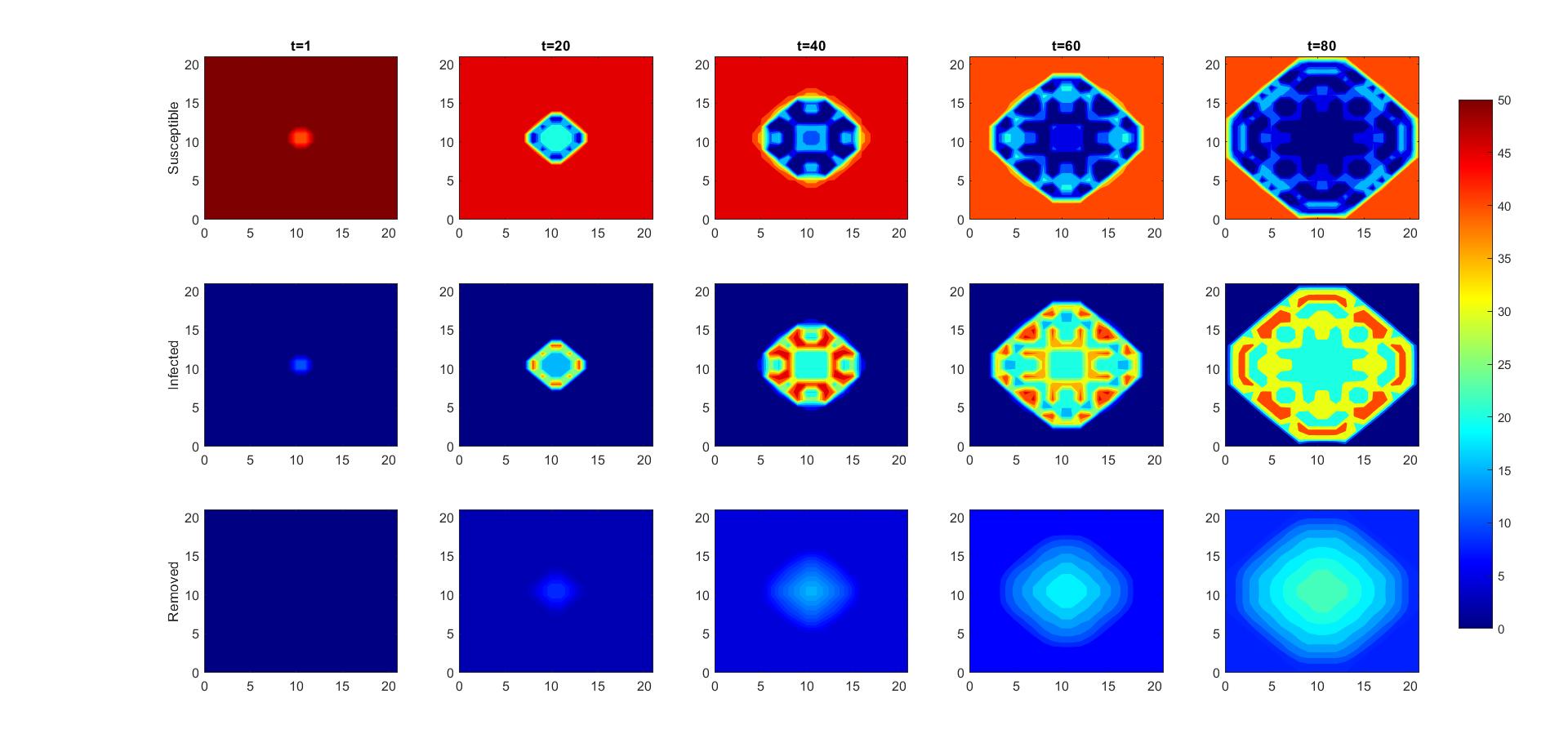}
\caption{For $p=15$}\label{F6A}
\end{subfigure}
\hfill
\begin{subfigure}[t]{0.32\textwidth}
\includegraphics[scale=0.3, height=2.7cm]{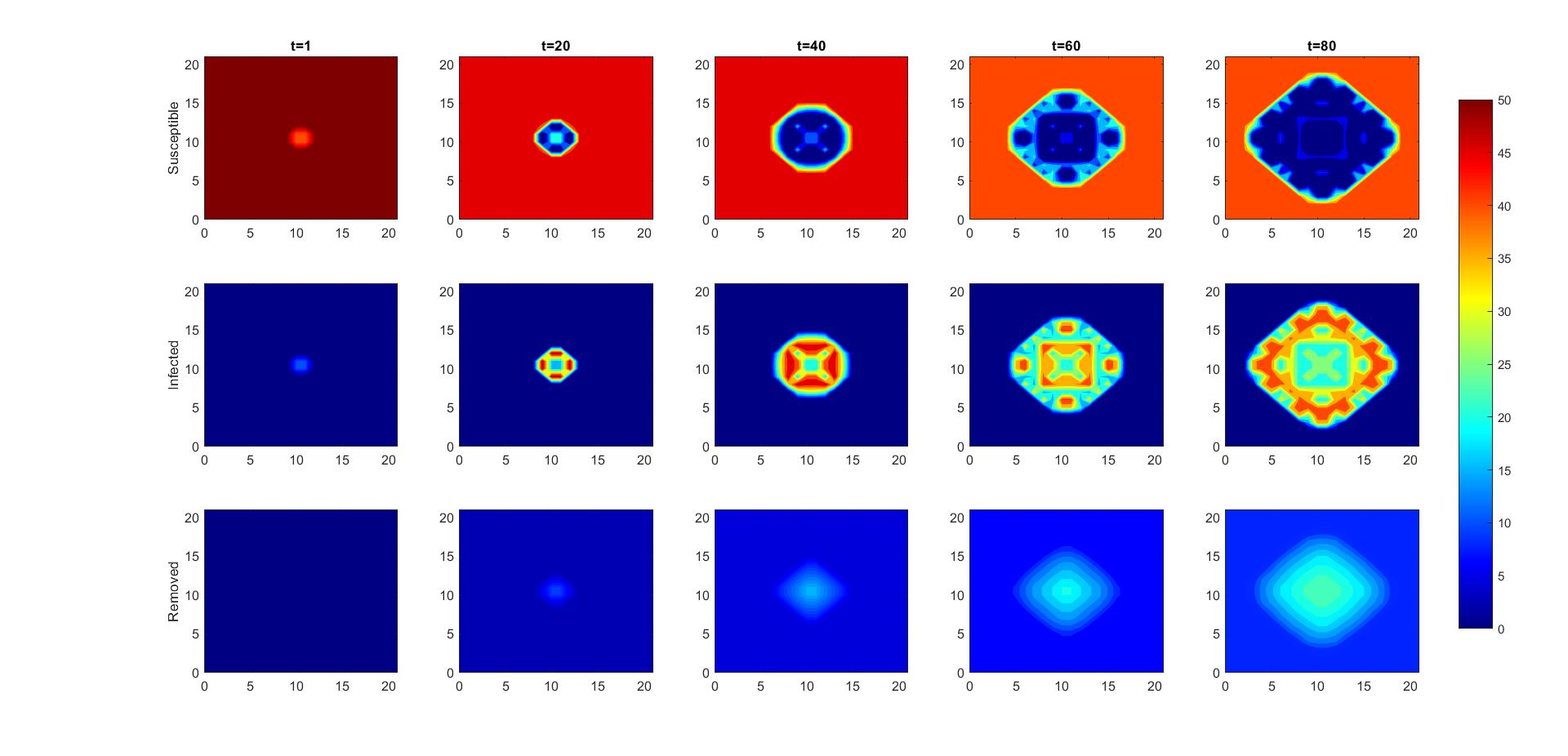}
\caption{For $p=10$}\label{F6B}
\end{subfigure}
\hfill
\begin{subfigure}[t]{0.32\textwidth}
\includegraphics[scale=0.3, height=2.7cm]{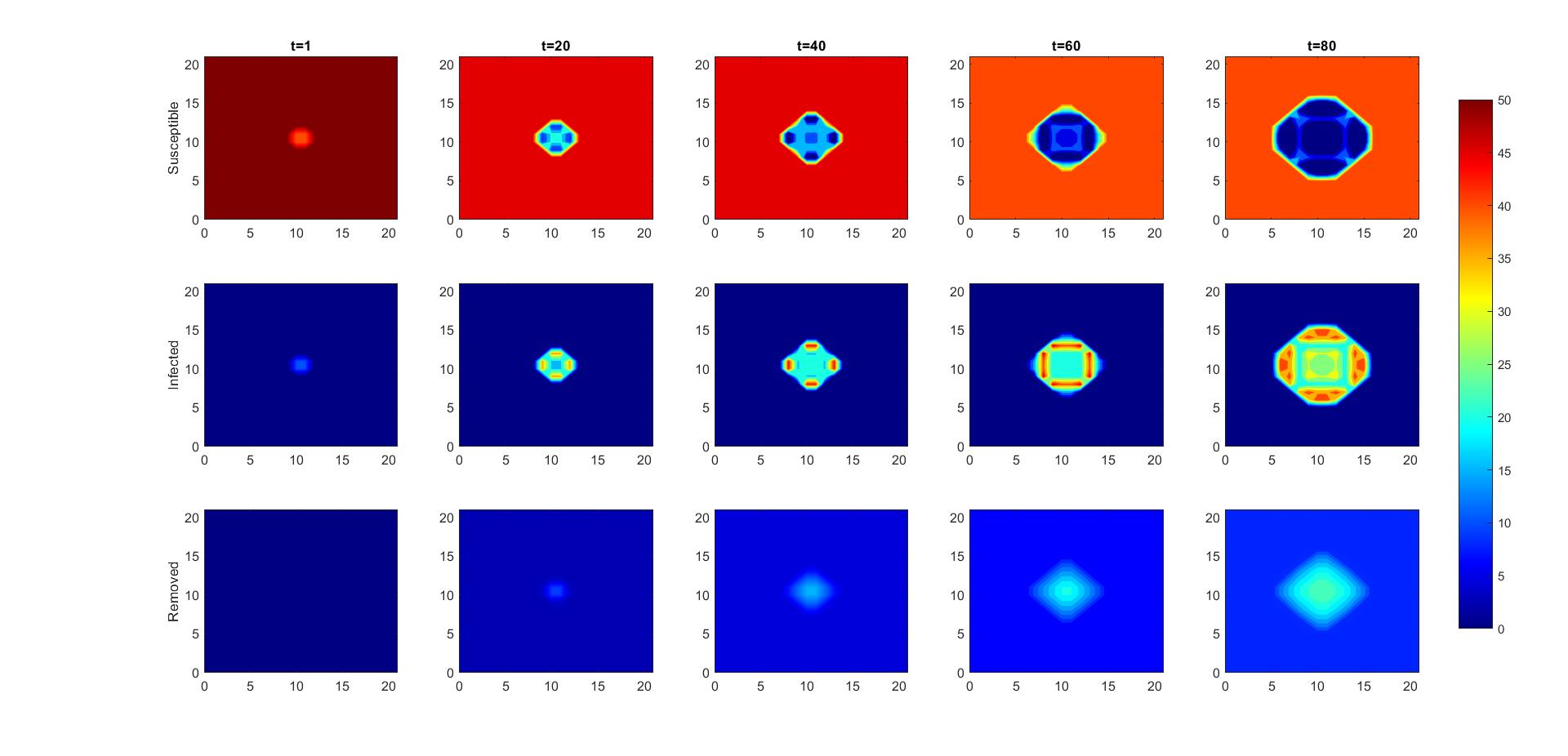}
\caption{For $p=5$}\label{F6C}
\end{subfigure}
\caption{Numerical findings of \eqref{E3.1}--\eqref{E3.3}
with vaccination program for $\alpha=1$.}\label{F6}
% ======================================
\end{sidewaysfigure}
\begin{sidewaysfigure}
% ======================================
\bigskip
\centering
\begin{subfigure}[t]{0.32\textwidth}
\includegraphics[scale=0.3, height=2.7cm]{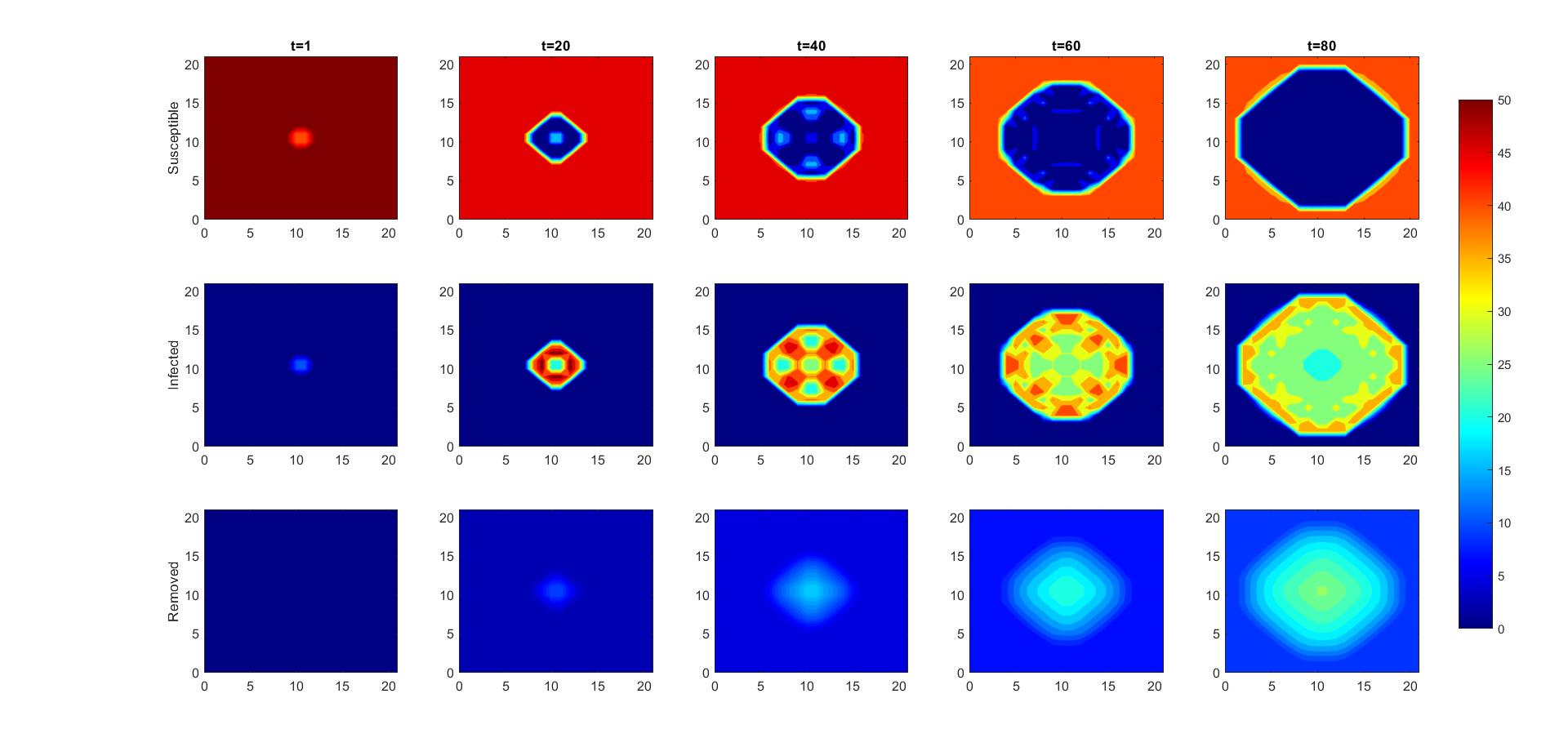}
\caption{For $p=15$}\label{F7A}
\end{subfigure}
\hfill
\begin{subfigure}[t]{0.32\textwidth}
\includegraphics[scale=0.3, height=2.7cm]{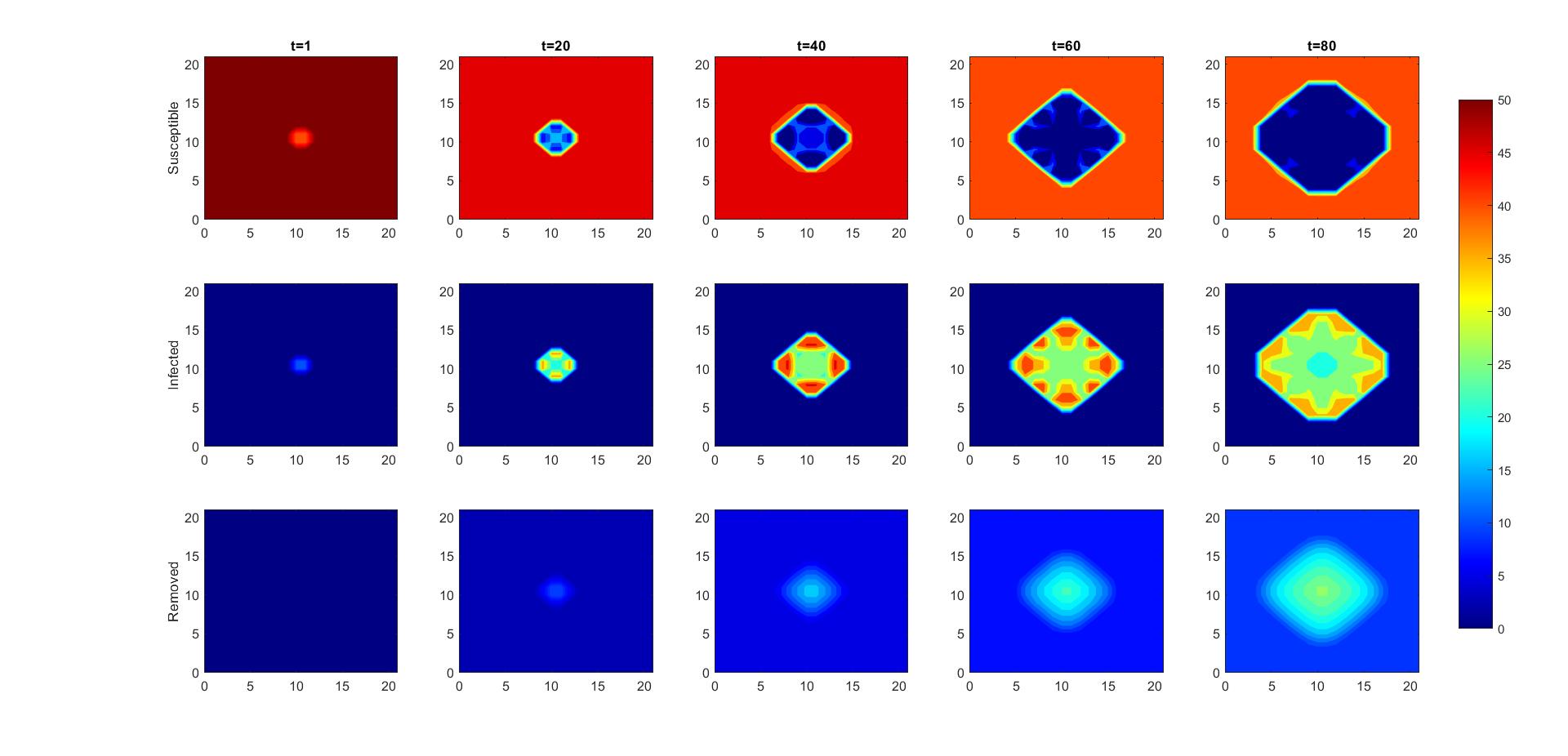}
\caption{For $p=10$}\label{F7B}
\end{subfigure}
\hfill
\begin{subfigure}[t]{0.32\textwidth}
\includegraphics[scale=0.3, height=2.7cm]{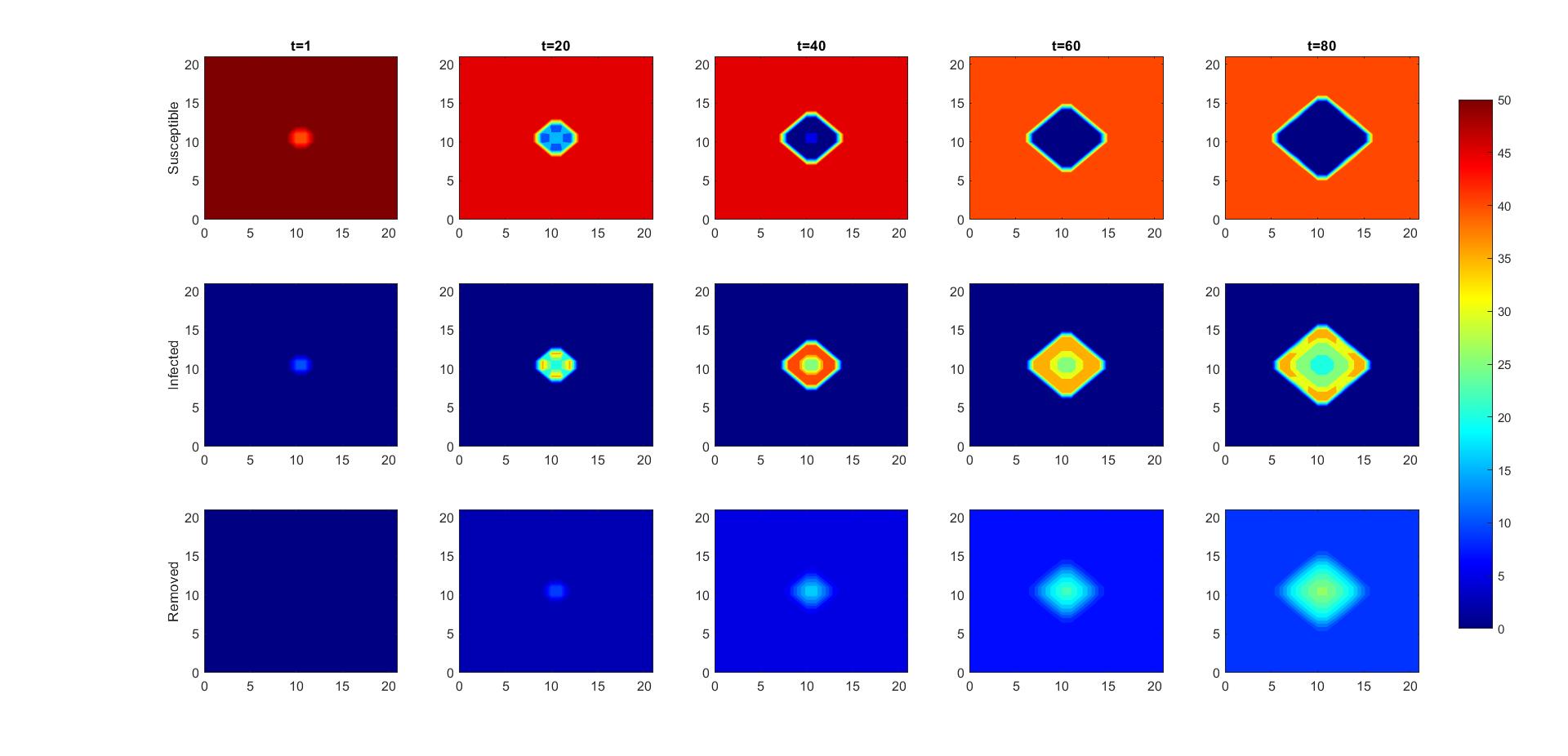}
\caption{For $p=5$}\label{F7C}
\end{subfigure}
\caption{Numerical findings of \eqref{E3.1}--\eqref{E3.3}
with vaccination program for $\alpha=0.95$.}\label{F7}
% ======================================
\vspace{0.2cm}
% ======================================
\begin{subfigure}[t]{0.32\textwidth}
\includegraphics[scale=0.3,height=2.7cm]{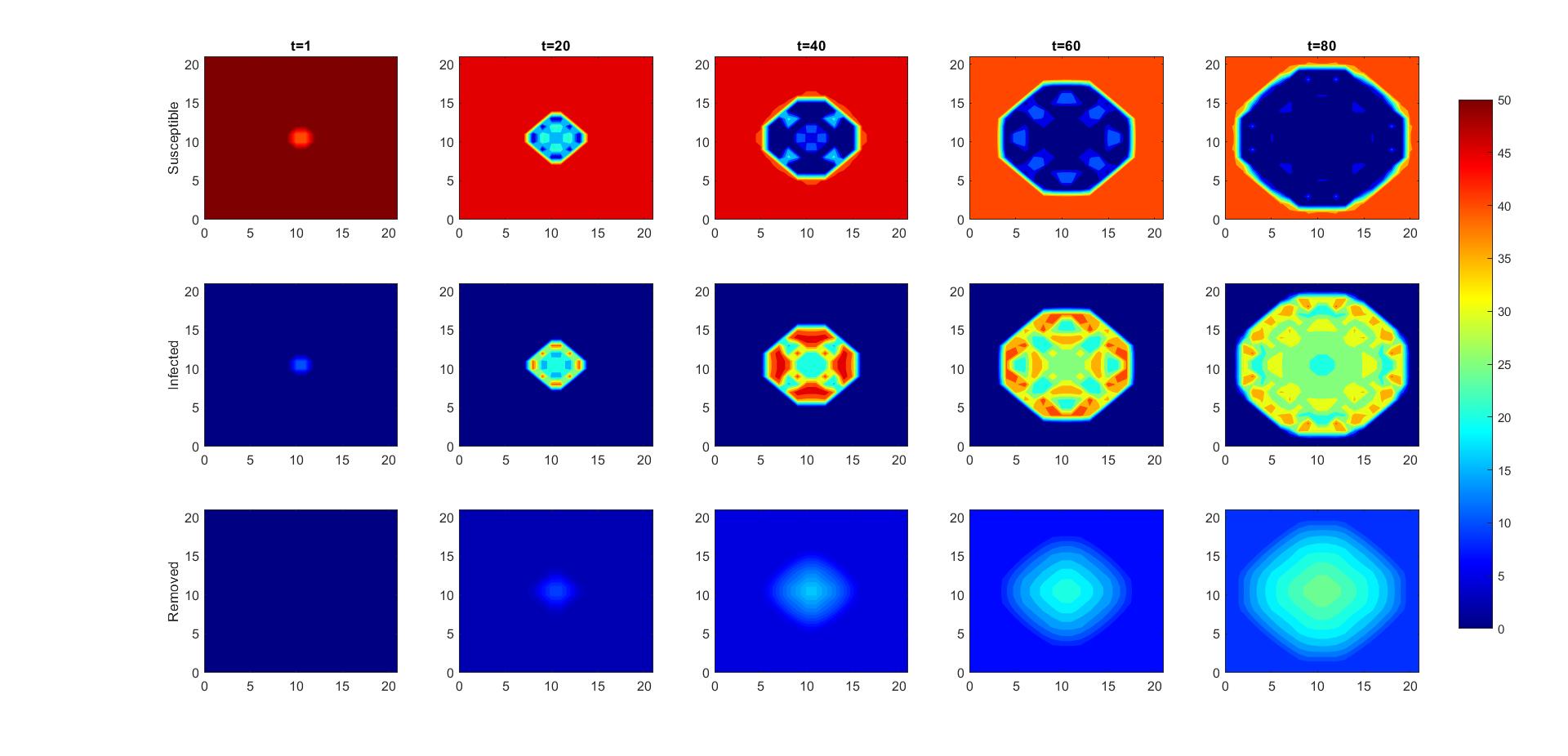}
\caption{For $p=15$}\label{F8A}
\end{subfigure}
\hfill
\begin{subfigure}[t]{0.32\textwidth}
\includegraphics[scale=0.3,height=2.7cm]{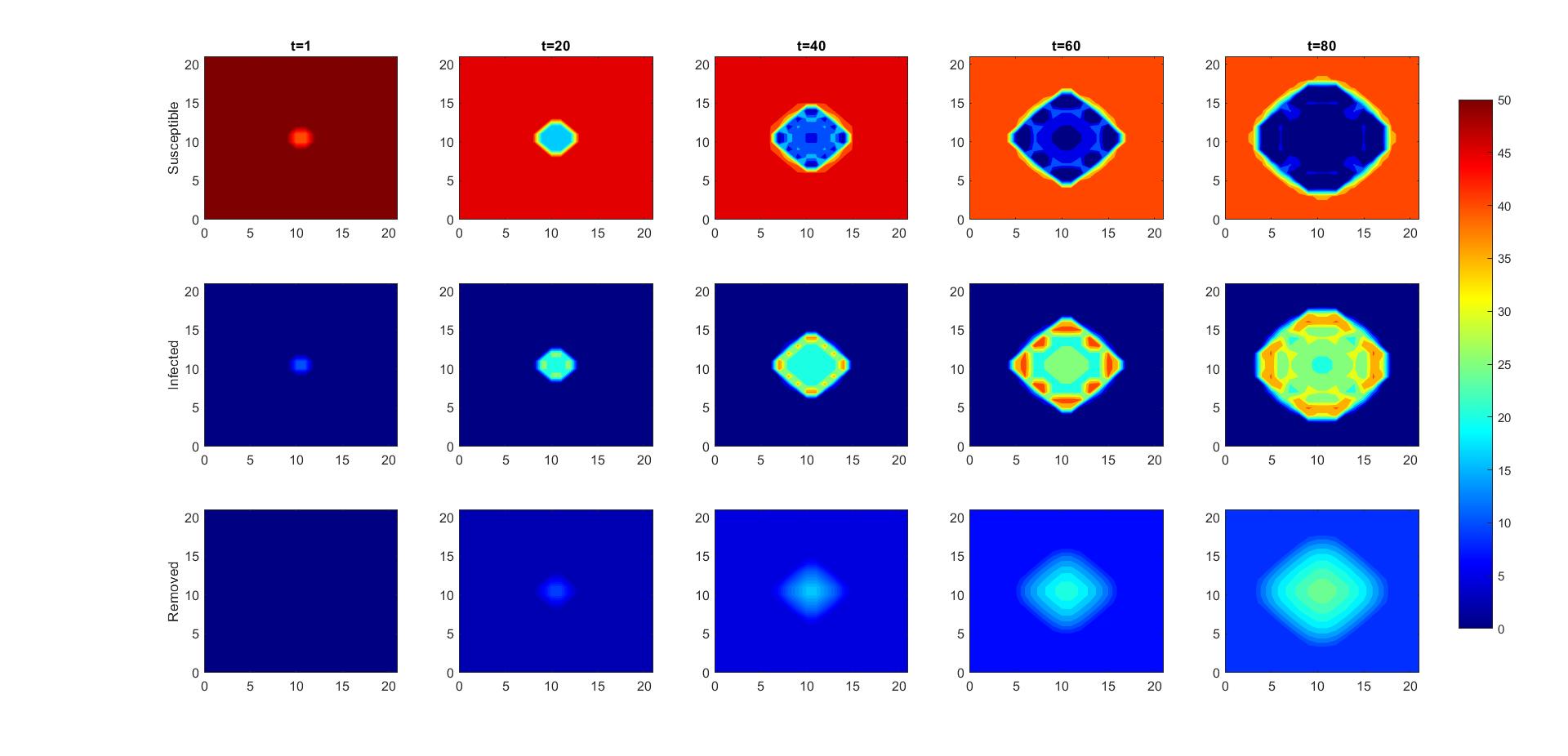}
\caption{For $p=10$}\label{F8B}
\end{subfigure}
\hfill
\begin{subfigure}[t]{0.32\textwidth}
\includegraphics[scale=0.3,height=2.7cm]{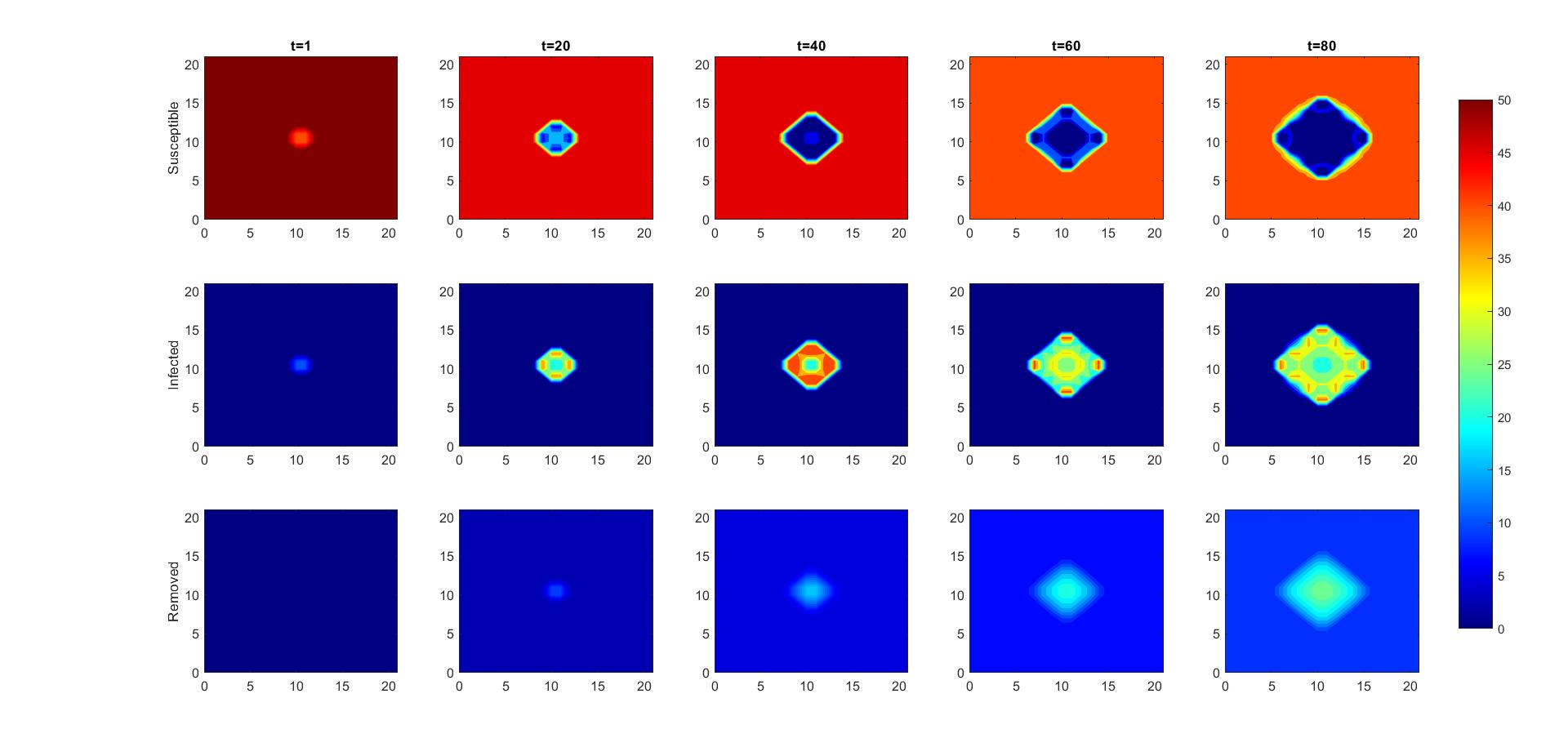}
\caption{For $p=5$}\label{F8C}
\end{subfigure}
\caption{Numerical findings of \eqref{E3.1}--\eqref{E3.3}
with vaccination program for $\alpha=0.9$.}\label{F8}
% ======================================
\vspace{0.2cm}
% ======================================
\begin{subfigure}[t]{0.32\textwidth}
\centering
\href{https://drive.google.com/drive/u/1/folders/1VyyfeWp9svoP6c_pKcNzpkvh7aXdpxO5}{%
\includegraphics[scale=0.28]{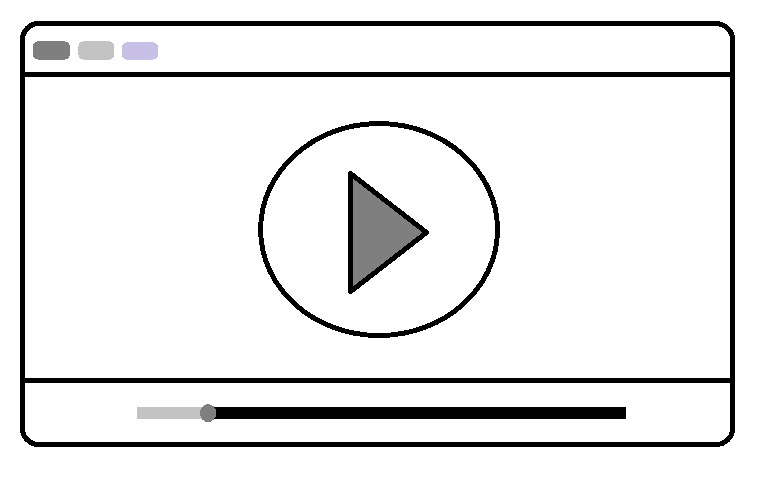}}
\caption{For $\alpha=1$ and $p=15, 10, 5$}\label{F9A}
\end{subfigure}
\hfill
\begin{subfigure}[t]{0.32\textwidth}
\centering
\href{https://drive.google.com/drive/u/1/folders/1W_qInlmDEhVcqk-wxraKydRlq0HbaLJd}{%
\includegraphics[scale=0.28]{System_dynamics.png}}
\caption{For $\alpha=0.95$ and $p=15, 10, 5$}\label{F9B}
\end{subfigure}
\hfill
\begin{subfigure}[t]{0.32\textwidth}
\centering
\href{https://drive.google.com/drive/u/1/folders/1MF9bbTxN5aIRgQgEmgY7-xPlDgC9WfdM}{%
\includegraphics[scale=0.28]{System_dynamics.png}}
\caption{For $\alpha=0.9$ and $p=15, 10, 5$}\label{F9C}
\end{subfigure}
\caption{System dynamics (Videos) with vaccination strategy in 80 days.}\label{F9}
% ======================================
\end{sidewaysfigure}

\newpage

% -------------------------------------------------

\section{Conclusion}
\label{S8}

In this research work, we have introduced a new innovative application 
of spatiotemporal models within the framework of optimal control theory. 
Our approach involves the modeling of interactions among the three compartments 
through a system of fractional equations, utilizing the Caputo fractional 
derivative and the $p$-Laplacian operator.
We claim that our study has the potential to yield more realistic models 
for the spread of diseases in specific scenarios. One of the key contributions 
is the establishment of the existence and uniqueness of a solution for our 
biological system, as well as the derivation of an optimal control program.
We have characterized the optimal control by making use of the associated state 
and adjoint variables, shedding light on the effective management of the disease spread.
Moreover, as part of our illustrative applications, we have conducted extensive 
simulations to observe how infectious diseases propagate and how they can be 
controlled. The results of these simulations clearly demonstrate that an increase 
in the value of $p$ or the choice of integer values for $\alpha$ 
leads to a more rapid spread of the disease.
Notably, our research highlights that the implementation of a vaccination 
strategy is highly successful in controlling the spread of infections during 
the designated time frame. These findings are of considerable significance 
for public health and epidemiological research, as they contribute to our 
understanding of how different parameters influence the dynamics of infectious diseases.

In this study, we focused on the SIR model as a foundational framework 
to investigate the impact of fractional operators and the $p$-Laplacian 
within a reaction-diffusion system. While the SIR model is widely regarded 
as a benchmark for studying epidemic dynamics, we acknowledge that more 
complex models, such as the SEIR/SIQR models or multi-compartmental systems, 
may capture additional epidemiological features. Incorporating such models 
could provide deeper insights into specific scenarios, such as latent 
infection stages or age-structured populations. Extending our analysis 
to these more intricate frameworks would be a valuable direction 
for future research. It could further demonstrate the versatility 
and applicability of fractional reaction-diffusion systems in epidemiology.

% -------------------------------------------------

\section*{Acknowledgments}

The authors express their appreciation to the reviewers 
for their valuable, constructive comments, and suggestions.

% -------------------------------------------------

\section*{Declarations}

% --------------------------------------------------------------------

\subsection*{Funding}

This work is part of Zinihi's PhD thesis done under the financial support of
CNRST as part of the PASS program. Torres was supported by 
Funda\c{c}\~{a}o para a Ci\^{e}ncia e a Tecnologia (FCT),
within project number UIDB/04106/2020 (CIDMA).

% --------------------------------------------------------------------

\subsection*{Data availability}

All information analyzed or generated, which would support the results
of this work are available in this article. No data was used
for the research described in the article.

% --------------------------------------------------------------------

\subsection*{Conflict of interest}

The authors declare that there are no problems or conflicts
of interest between them that may affect the study in this paper

% -------------------------------------------------

% -------------------------------------------------

\medskip

Received October 2024; revised January 2025; early access February 2025.

\medskip

\end{document}